\documentclass[11pt]{amsart}
\usepackage{graphicx,dsfont}
\usepackage[a4paper,margin=3cm]{geometry}

\theoremstyle{plain}
\newtheorem{theorem}{Theorem}[section] 
\newtheorem{proposition}[theorem]{Proposition}                          
\newtheorem{lemma}[theorem]{Lemma}
\newtheorem{corollary}[theorem]{Corollary}

\theoremstyle{definition}

\theoremstyle{remark}
\newtheorem{remark}[theorem]{Remark}

\begin{document}

\date{Preprint, accepted in ALEA, March 2010}
\keywords{random matrices, 
  reversible Markov chains, 
  random walks, random environment, spectral gap, Wigner's
  semi--circle law, arc--sine law, tridiagonal matrices, birth-and-death
  processes, spectral analysis, 
  homogenization.} 
\subjclass{15A52; 60K37; 60F15; 62H99; 37H10; 47B36.}

\author{Charles~Bordenave} \address{Institut de Math\'ematiques de Toulouse
  UMR5219 CNRS Universit\'e de Toulouse, France}
\email{charles.bordenave(at)math.univ-toulouse.fr}

\author{Pietro~Caputo}
\address{Dipartimento di Matematica, Universit\`a Roma Tre, Italy} 
\email{caputo(at)mat.uniroma3.it}

\author{Djalil~Chafa{\"{\i}}} 
\address{LAMA UMR8050 CNRS Universit\'e
  Paris-Est Marne-la-Vall\'ee, France} 
\email{djalil(at)chafai.net}

\title[Spectrum of large random reversible Markov chains]{Spectrum of large random reversible Markov chains: \\ two examples} 

\begin{abstract}
  We take on a Random Matrix theory viewpoint to study the spectrum of certain
  reversible Markov chains in random environment. As the number of states
  tends to infinity, we consider the global behavior of the spectrum, and the
  local behavior at the edge, including the so called spectral gap. Results
  are obtained for two simple models with distinct limiting features. The
  first model is built on the complete graph while the second is a
  birth-and-death dynamics. Both models give rise to random matrices with non
  independent entries.
\end{abstract}

\maketitle

\newcommand{\wt}{\widetilde} 
\newcommand{\scalar}[2]{\langle #1 , #2\rangle}
\newcommand{\what}[1][]{\ensuremath{\clubsuit\otimes\clubsuit}\texttt{ #1}}

\section{Introduction}

The spectral analysis of large dimensional random matrices is a very active
domain of research, connected to a remarkable number of areas of Mathematics,
see e.g.\
\cite{MR2129906,MR1746976,MR1711663,MR1864966,anderson-gionnet-zeitouni,MR2432537}.
On the other hand, it is well known that the spectrum of reversible Markov
chains provides useful information on their trend to equilibrium, see e.g.\
\cite{MR1490046,MR2166276,tetali-montenegro,MR2466937}. The aim of this paper
is to explore potentially fruitful links between the Random Matrix and the
Markov Chains literature, by studying the spectrum of reversible Markov chains
with large finite state space in a frozen random environment. The latter is
obtained by assigning random weights to the edges of a finite graph. This
approach raises a collection of stimulating problems, lying at the interface
between Random Matrix theory, Random Walks in Random Environment, and Random
Graphs. We focus here on two elementary models with totally different scalings
and limiting objects: a complete graph model and a chain graph model. The
study of spectral aspects of random Markov chains or random walks in random
environment is not new, see for instance
\cite{cheliotis-virag,MR1956471,MR2071631,MR2370603,MR2152251,MR1890289,MR2198849}
and references therein. Here we adopt a Random Matrix theory point of view.

Consider a finite connected undirected graph $G=(V,E)$, with vertex set $V$
and edge set $E$, together with a set of weights, given by nonnegative random
variables
$$
\mathbf{U}=\{U_{i,j};\{i,j\}\in E\}.
$$
Since the graph $G$ is undirected we set $U_{i,j}=U_{j,i}$. On the network
$(G,\mathbf{U})$, we consider the random walk in random environment with state
space $V$ and transition probabilities
\begin{equation}\label{eq:srw}
K_{i,j} = \frac{U_{i,j}}{\rho_i}
\quad\text{where}\quad\rho_i=\sum_{j:\{i,j\}\in E} U_{i,j}.
\end{equation}
The Markov kernel $K$ is reversible with respect to the measure
$\rho=\{\rho_i\,,\;i\in V\}$ in that 
$$
\rho_i K_{i,j} = \rho_j K_{j,i}
$$
for all $i,j\in V$. When the variables $\mathbf{U}$ are all equal to a
positive constant this is just the standard simple random walk on $G$, and
$K-I$ is the associated Laplacian. If $\rho_{i_0}=0$ for some vertex $i_0$
then we set $K_{i_0,j}=0$ for all $j\neq i_0$ and $K_{i_0,i_0}=1$ ($i_0$ is
then an isolated vertex).

The construction of reversible Markov kernels from graphs with weighted edges
as in \eqref{eq:srw} is classical in the Markovian literature, see e.g.\
\cite{MR2166276,MR920811}. As for the choice of the graph $G$, we shall work
with the simplest cases, namely the complete graph or a one--dimensional chain
graph. Before passing to the precise description of models and results, let us
briefly recall some broad facts.

By labeling the $n=|V|$ vertices of $G$ and putting $K_{i,j}=0$ if
$\{i,j\}\not\in E$, one has that $K$ is a random $n\times n$ Markov matrix.
The entries of $K$ belong to $[0,1]$ and each row sums up to $1$. The spectrum
of $K$ does not depend on the way we label $V$. In general, even if the random
weights $\mathbf{U}$ are i.i.d.\ the random matrix $K$ has non--independent
entries due to the normalizing sums $\rho_i$. Note that $K$ is in general
non--symmetric, but by reversibility, it is symmetric w.r.t.\ the scalar product induced by $\rho$, 
and its
spectrum $\sigma(K)$ is real. Moreover, $1\in\sigma(K)\subset[-1,+1]$, and it
is convenient to denote the eigenvalues of $K$ by
$$
-1\leq\lambda_n(K)\leq\cdots\leq\lambda_1(K)=1.
$$
If the weights $U_{i,j}$ are all positive, then $K$ is irreducible, the
eigenspace of the largest eigenvalue $1$ is one--dimensional and thus
$\lambda_2(K)<1$. In this case $\rho_i$ is its unique invariant distribution,
up to normalization. Moreover, since $K$ is reversible, the period of $K$ is
$1$ (aperiodic case) or $2$, and this last case is equivalent to
$\lambda_n(K)=-1$ (the spectrum of $K$ is in fact symmetric when $K$ has
period $2$); see e.g.\ \cite{MR2209438}. 

The \emph{bulk behavior} of $\sigma(K)$ is studied via the Empirical Spectral
Distribution (ESD)
$$
\mu_K=\frac1n\sum_{k=1}^n \delta_{\lambda_k(K)}.
$$
Since $K$ is Markov, its ESD contains probabilistic information on the
corresponding random walk. Namely, the moments of the ESD $\mu_K$ satisfy, for
any $\ell\in\mathds{Z}_+$
\begin{equation}\label{eq:moms}
  \int_{-1}^{+1}\!x^\ell\mu_K(dx) %
  = \frac1n \mathrm{Tr}(K^\ell)= \frac1n\sum_{i\in V} r^\mathbf{U}_\ell(i)
\end{equation}
where $r^\mathbf{U}_\ell(i)$ denotes the probability that the random walk on
$(G,{\bf U})$ started at $i$ returns to $i$ after $\ell$ steps. 

The \emph{edge behavior} of $\sigma(K)$ corresponds to the extreme eigenvalues
$\lambda_2(K)$ and $\lambda_n(K)$, or more generally, to the $k$--extreme
eigenvalues $\lambda_2(K),\dots,\lambda_{k+1}(K)$ and
$\lambda_{n}(K),\dots,\lambda_{n-k+1}(K)$. The geometric decay to the
equilibrium measure $\rho$ of the \emph{continuous time} random walk with
semigroup ${(e^{t(K-I)})}_{t\geq0}$ generated by $K-I$ is governed by the so
called \emph{spectral gap}
$$
\mathrm{gap}(K-I) = 1-\lambda_2(K).
$$
In the aperiodic case, the relevant quantity for the \emph{discrete time}
random walk with kernel $K$ is
$$
\varsigma(K)=1-\max_{\substack{\lambda\in\sigma(K)\\\lambda\neq1}}|\lambda|
= 1-\max(-\lambda_n(K),\lambda_2(K))\,.
$$
In that case, for any fixed value of $n$, we have $(K^\ell)_{i,\cdot}\to\rho$
as $\ell\to\infty$, for every $1\leq i\leq n$. We refer to e.g.\
\cite{MR1490046,MR2466937} for more details.

\subsection*{Complete graph model}

Here we set $V=\{1,\dots,n\}$ and $E=\{\{i,j\};i,j\in V\}$. Note that we have
a loop at any vertex. The weights $U_{i,j}$, $1\leq i\leq j\leq n$ are i.i.d.\
random variables with common law $\mathcal{L}$ supported on $[0,\infty)$. The
law $\mathcal{L}$ is independent of $n$. Without loss of generality, we assume
that the marks $\mathbf{U}$ come from the truncation of a single infinite
triangular array $(U_{i,j})_{1\leq i\leq j}$ of i.i.d.\ random variables of
law $\mathcal{L}$. This defines a common probability space, which is
convenient for almost sure convergence as $n\to\infty$.

When $\mathcal{L}$ has finite mean $\int_0^\infty x\,\mathcal{L}(dx) = m$ we
set $m=1$. This is no loss of generality since $K$ is invariant under the
linear scaling $t\to t\,U_{i,j}$. If $\mathcal{L}$ has a finite second moment
we write $\sigma^2 = \int_0^\infty(x-1)^2\,\mathcal{L}(dx)$ for the variance.
The rows of $K$ are equally distributed (but not independent) and follow an
exchangeable law on $\mathds{R}^n$. Since each row sums up to one, we get by
exchangeability that for every $1\leq i,j\neq j'\leq n$,
$$
\mathds{E}(K_{i,j})=\frac{1}{n} %
\quad\text{and}\quad %
\mathrm{Cov}(K_{i,j},K_{i,j'})=-\frac{1}{n-1}\mathrm{Var}(K_{1,1}).
$$
Note that $\mathcal{L}$ may have an atom at $0$, i.e.\
$\mathds{P}(U_{i,j}=0)=1-p$, for some $p\in(0,1)$. In this case $K$ describes
a random walk on a weighted version of the standard Erd\H{o}s-R\'enyi $G(n,p)$
random graph. Since $p$ is fixed, almost surely (for $n$ large enough) there
is no isolated vertex, the row-sums $\rho_i$ are all positive, and $K$ is
irreducible.

The following theorem states that if $\mathcal{L}$ has finite positive
variance $0<\sigma^2<\infty$, then the bulk of the spectrum of $\sqrt{n}K$
behaves as if we had a Wigner matrix with 
i.i.d.\ entries, i.e.\ as if
$\rho_i\equiv n$. 
We refer to e.g.\ \cite{MR1711663,anderson-gionnet-zeitouni} for more on Wigner matrices and the semi--circle law. 
The ESD of $\sqrt{n}K$ is
$\mu_{\sqrt{n}K}=\frac{1}{n}\sum_{k=1}^n\delta_{\sqrt{n}\lambda_k(K)}.$ 

\begin{theorem}[Bulk behavior]\label{th:wigner}
  If $\mathcal{L}$ has finite positive variance $0<\sigma^2<\infty$ then 
  $$
  \mu_{\sqrt{n}K} %
  \underset{n\to\infty}{\overset{w}{\longrightarrow}} %
  \mathcal{W}_{2\sigma}
  $$
  almost surely, where ``$\overset{w}{\to}$'' stands for weak convergence of
  probability measures and $\mathcal{W}_{2\sigma}$ is Wigner's semi--circle
  law with Lebesgue density
  \begin{equation}\label{semi}
  x\mapsto
  \frac{1}{2\pi\sigma^2}\sqrt{4\sigma^2-x^2}
  \,\,\mathds{1}_{\left[-2\sigma,+2\sigma\right]}(x)\,.
  \end{equation}
\end{theorem}

The proof of Theorem \ref{th:wigner}, given in Section \ref{se:wigner}, relies
on a uniform strong law of large numbers which allows to estimate $\rho_i =
n(1+o(1))$ and therefore yields a comparison of $\sqrt{n}K$ with a suitable
Wigner matrix with i.i.d.\  entries. Note that, even though
\begin{equation}\label{eq:lamax}
  \lambda_1(\sqrt{n}K)=\sqrt{n}\to\infty \quad\text{as}\quad n\to\infty,
\end{equation}
the weak limit of $\mu_{\sqrt{n}K}$ is not affected since
$\lambda_1(\sqrt{n}K)$ has weight $1/n$ in $\mu_{\sqrt{n}K}$. Theorem
\ref{th:wigner} implies that the bulk of $\sigma(K)$ collapses 
weakly 
at speed $n^{-1/2}$. Concerning the extremal eigenvalues
$\lambda_n(\sqrt{n}K)$ and $\lambda_2(\sqrt{n}K)$, we only get from Theorem
\ref{th:wigner} that almost surely, for every fixed $k\in\mathds{Z}_+$,
$$
\liminf_{n\to\infty}\sqrt{n}\lambda_{n-k}(K) \leq -2\sigma %
\quad\text{and}\quad %
\limsup_{n\to\infty}\sqrt{n}\lambda_{k+2}(K) \geq +2\sigma.
$$
The result below gives the behavior of the extremal eigenvalues under the
assumption that $\mathcal{L}$ has finite fourth moment (i.e.\
$\mathds{E}(U_{1,1}^4)<\infty$).

\begin{theorem}[Edge behavior]\label{th:wigneredge}
  If $\mathcal{L}$ has finite positive variance $0<\sigma^2<\infty$ and finite
  fourth moment then almost surely, for any fixed $k\in\mathds{Z}_+$,
  $$
  \lim_{n\to\infty}\sqrt{n}\lambda_{n-k}(K)=-2\sigma
  \quad\text{and}\quad
  \lim_{n\to\infty}\sqrt{n}\lambda_{k+2}(K)=+2\sigma.
  $$
  In particular, almost surely,
  \begin{equation}\label{edge}
  \mathrm{gap}(K-I)=1-\frac{2\sigma}{\sqrt{n}}+o\left(\frac{1}{\sqrt{n}}\right)
  \quad\text{and}\quad
  \varsigma(K)=1-\frac{2\sigma}{\sqrt{n}}+o\left(\frac{1}{\sqrt{n}}\right).
  \end{equation}
\end{theorem}

The proof of Theorem \ref{th:wigneredge}, given in Section \ref{se:wigner},
relies on a suitable rank one reduction which allows us to compare
$\lambda_2(\sqrt{n}K)$ with the largest eigenvalue of a Wigner matrix with
centered entries. This approach also requires a refined version of the uniform
law of large numbers used in the proof of Theorem \ref{th:wigner}.

The edge behavior of Theorem \ref{th:wigneredge} allows one to reinforce 
Theorem \ref{th:wigner} by providing convergence of moments. Recall that for any integer $p\geq1$, the weak
convergence together with the convergence of moments up to order $p$ is
equivalent to the convergence in Wasserstein $W_p$ distance, see e.g.\ 
\cite{MR1964483}. For every real $p\geq1$, the Wasserstein distance
$W_p(\mu,\nu)$ between two probability measures $\mu,\nu$ on $\mathds{R}$ is
defined by
\begin{equation}\label{eq:Wp}
  W_p(\mu,\nu) = \inf_{\Pi}%
  \left(\int_{\mathds{R}\times\mathds{R}}\!|x-y|^p\,\Pi(dx,dy)\right)^{1/p}
\end{equation}
where the infimum runs over the convex set of probability measures on
$\mathds{R}^2=\mathds{R}\times\mathds{R}$ with marginals $\mu$ and $\nu$. Let
$\wt\mu_{\sqrt{n}K}$ be the trimmed ESD defined by
$$
\wt\mu_{\sqrt{n}K} %
 = \frac{1}{n-1}\sum_{k=2}^n\delta_{\sqrt{n}\lambda_k(K)} %
 = \frac{n}{n-1}\mu_{\sqrt{n}K}-\frac{1}{n-1}\delta_{\sqrt{n}}.
$$
We have then the following Corollary of theorems \ref{th:wigner} and
\ref{th:wigneredge}, proved in Section \ref{se:wigner}.

\begin{corollary}[Strong convergence]\label{co:wignerstrong}
  If $\mathcal{L}$ has positive variance and finite fourth moment then
  almost surely, for every $p\geq1$,
  $$
  \lim_{n\to\infty} W_p(\wt\mu_{\sqrt{n}K},\mathcal{W}_{2\sigma}) = 0%
  \quad\text{while}\quad %
  \lim_{n\to\infty} W_p(\mu_{\sqrt{n}K},\mathcal{W}_{2\sigma}) %
  = %
  \begin{cases}
    0 & \text{if $p<2$} \\
    1 & \text{if $p=2$} \\
    \infty & \text{if $p>2$}.
  \end{cases}
  $$
\end{corollary}

Recall that for every $k\in\mathds{Z}_+$, the $k^\text{th}$ moment of the
semi--circle law $\mathcal{W}_{2\sigma}$ is zero if $k$ is odd and is
$\sigma^k$ times the $(k/2)^\text{th}$ Catalan number if $k$ is even. 
The $r^\text{th}$ Catalan number $\frac{1}{r+1}\,\binom{2r}{r}$ counts,
among other things, the number of
non--negative simple paths of length $2r$ that start and end at $0$.

On the
other hand, from \eqref{eq:moms}, we know that for every $k\in\mathds{Z}_+$, the
$k^\text{th}$ moment of the ESD $\mu_{\sqrt{n}K}$ writes
$$
\int_\mathds{R}\!x^k\,\mu_{\sqrt{n}K}(dx) %
= \frac1n \,{\rm Tr}\left((\sqrt{n}K)^k\right) %
= n^{-1+\frac{k}2}\sum_{i=1}^n r^\mathbf{U}_k(i)\,.
$$
Additionally, from \eqref{eq:lamax} we get
$$
\int_\mathds{R}\!x^k\,\mu_{\sqrt{n}K}(dx) %
=
n^{-1+\frac{k}{2}} %
+ \left(1-\frac{1}{n}\right)\int_\mathds{R}\!x^k\,\wt\mu_{\sqrt{n}K}(dx)
$$
where $\wt\mu_{\sqrt{n}K}$ is the trimmed ESD defined earlier. We can then
state the following.

\begin{corollary}[Return probabilities]\label{co:retpro}
  Let $r^\mathbf{U}_k(i)$ be the probability that the random walk on $V$ with
  kernel $K$ started at $i$ returns to $i$ after $k$ steps. If $\mathcal{L}$
  has variance $0<\sigma^2<\infty$ and finite fourth moment
  then almost surely, for every $k\in\mathds{Z}_+$,
  \begin{equation}\label{moments}
    \lim_{n\to\infty} n^{-1+\frac{k}2}\left(\sum_{i=1}^nr^\mathbf{U}_k(i)-1\right) %
    = %
    \begin{cases}
      0 & \text{if $k$ is odd}\\
      \frac{\sigma^{k}}{k/2+1}\binom{k}{k/2} & \text{if $k$ is even}.
    \end{cases}
  \end{equation} 
\end{corollary}

We end  our analysis of the complete graph model with the behavior of the
invariant probability distribution $\hat{\rho}$ of $K$, obtained by
normalizing the invariant vector $\rho$ as
$$
\hat{\rho}=(\rho_1+\cdots+\rho_n)^{-1}(\rho_1\delta_1+\cdots+\rho_n\delta_n).
$$
Let $\mathcal{U}=n^{-1}(\delta_1+\cdots+\delta_n)$ denote the uniform law on
$\{1,\ldots,n\}$. As usual, the \emph{total variation distance}
$\|\mu-\nu\|_\textsc{tv}$ between two probability measures
$\mu=\sum_{k=1}^n\mu_k\delta_k$ and $\nu=\sum_{k=1}^n\nu_k\delta_k$ on
$\{1,\ldots,n\}$ is given by
$$
\|\mu-\nu\|_\textsc{tv}=\frac{1}{2}\sum_{k=1}^n|\mu_k-\nu_k|.
$$

\begin{proposition}[Invariant probability measure]\label{th:wignerinv}
  If $\mathcal{L}$ has finite second moment, then a.s.
  \begin{equation}\label{tvar}
    \lim_{n\to\infty}\|\hat\rho - \mathcal{U}\|_{\textsc{tv}}=0.
  \end{equation}
\end{proposition}

The proof of Proposition \ref{th:wignerinv}, given in Section \ref{se:wigner},
relies as before on a uniform law of large numbers. The speed of convergence
and fluctuation of $\|\hat\rho - \mathcal{U}\|_{\textsc{tv}}$ depends on the
tail of $\mathcal{L}$. The reader can find in Lemma \ref{le:LLN} of Section
\ref{se:wigner} some estimates in this direction.

\subsection*{Chain graph model (birth-and-death)}

The complete graph model discussed earlier provides a random reversible Markov
kernel which is irreducible and aperiodic. One of the key feature of this
model lies in the fact that the degree of each vertex is $n$, which goes to
infinity as $n\to\infty$. This property allows one 
to use a law of large numbers
to control the normalization $\rho_i$. The method will roughly still work if
we replace the complete graphs sequence by a sequence of graphs for which the
degrees are of order $n$. See e.g.\ \cite{MR2432537} for a survey of related
results in the context of random graphs. To go beyond this framework, it is
natural to consider \emph{local} models 
for which the degrees are uniformly bounded. We shall focus on a simple
birth-and-death Markov kernel $K = (K_{i,j})_{1\leq i,j\leq n}$ on
$\{1,\ldots,n\}$ given by
$$
K_{i , i+1} = b_i, \quad K_{i,i}=a_i, \quad K_{i,i-1} = c_i
$$ 
where $(a_i)_{1\leq i\leq n}$, $(b_i)_{1\leq i\leq n}$, $(c_i)_{1\leq i\leq
  n}$ are in $[0,1]$ with $c_1 = b_n =0$ , $b_i + a_i+ c_i = 1$ for every
$1\leq i\leq n$, and $c_{i+1}>0$ and $b_i > 0$ for every $1 \leq i \leq n-1$.
In other words, we have
\begin{equation}\label{eq:chainKabc}
K=
\begin{pmatrix}
  a_1 & b_1 \\
  c_2 & a_2 & b_2 \\
  & c_3 & a_3 & b_3 \\
  && \ddots & \ddots & \ddots \\
  &&& c_{n-1} & a_{n-1} & b_{n-1} \\
  &&&&   c_n & a_n
\end{pmatrix}.
\end{equation}
The kernel $K$ is irreducible, reversible, and every vertex has degree $\leq
3$. For an arbitrary $\rho_1>0$, the measure $\rho = \rho_1\delta
_1+\cdots+\rho_n\delta_n$ defined for every $2\leq i\leq n$ by
$$
\rho_i = \rho_1\prod_{k = 1 } ^{i-1} \frac{b_k}{c_{k+1}} %
=\rho_1\frac{b_1\cdots b_{i-1}}{c_2\cdots c_i} 
$$
is invariant and reversible for $K$, 
i.e.\ for $1 \leq i,j
\leq n$,
$
\rho_i K_{i,j} = \rho_j K_{j,i}.
$
For every $1\leq i\leq n$, the $i^\text{th}$ row $(c_i,a_i,b_i)$ of $K$
belongs to the $3$-dimensional simplex
$$
\Lambda_3=\{v\in[0,1]^3;v_1+v_2+v_3=1\}.
$$
For every $v\in\Lambda_3$, we define the left and right ``reflections''  
$v_-\in\Lambda_3$ and $v_+\in\Lambda_3$ of $v$ by
$$
v_-=(v_1+v_3,v_2,0)\quad\text{and}\quad v_+=(0,v_2,v_1+v_3).
$$
The following result provides a general answer for the behavior of the bulk.

\begin{theorem}[Global behavior for ergodic environment]\label{th:chainerg}
  Let $\mathbf{p}:\mathds{Z}\to\Lambda_3$ be an ergodic random field. Let $K$
  be the random birth-and-death kernel \eqref{eq:chainKabc} on
  $\{1,\ldots,n\}$ obtained from $\mathbf{p}$ by taking for every $1\leq i\leq
  n$
  $$
  (c_i,a_i,b_i)
  =
  \begin{cases}
    \mathbf{p}(i) & \text{if $2\leq i\leq n-1$} \\
    \mathbf{p}(1)_+ & \text{if $i=1$} \\
    \mathbf{p}(n)_- & \text{if $i=n$}.
  \end{cases}
  $$
  Then there exists a non-random probability measure $\mu$ on $[-1,+1]$ such
  that almost surely,
  $$
  \lim_{n\to\infty} W_p(\mu_K,\mu)=0
  $$
  for every $p\geq1$, where $W_p$ is the Wasserstein distance \eqref{eq:Wp}.
  Moreover, for every $\ell\geq0$,
  $$
  \int_{-1}^{+1}\!x^\ell\,\mu(dx)=\mathds{E}[r^\mathbf{p}_\ell(0)]
  $$
  where $r^\mathbf{p}_\ell(0)$ is the probability of return to $0$ in $\ell$
  steps for the random walk on $\mathds{Z}$ with random environment
  $\mathbf{p}$. The expectation is taken with respect to the environment
  $\mathbf{p}$.
\end{theorem}

The proof of Theorem \ref{th:chainerg}, given in Section \ref{se:chain}, is a
simple consequence of the ergodic theorem; see also \cite{MR1956471} for an
earlier application to random conductance models. The reflective boundary
condition is not necessary for this result on the bulk of the spectrum, and
essentially any boundary condition (e.g.\ Dirichlet or periodic) produces the
same limiting law, with essentially the same proof. Moreover, this result is
not limited to the one--dimensional random walks and it remains valid e.g.\
for any \emph{finite range reversible random walk with ergodic random
  environment on $\mathds{Z}^d$}. However, as we shall see below, a more
precise analysis is possible for certain type of environments when $d=1$.

Consider the chain graph $G=(V,E)$ with $V=\{1,\ldots,n\}$ and
$E=\{(i,j);|i-j|\leq 1\}$.
%
A
random conductance model on 
this
graph can be obtained by defining $K$ with \eqref{eq:srw} by putting i.i.d.\
positive weights $\mathbf{U}$ of law $\mathcal{L}$ on the
edges. 
For instance, if we remove the loops, 
this corresponds to define $K$ by \eqref{eq:chainKabc} with
$a_1=\cdots=a_n=0$, $b_1 = c_n = 1$, and, for every $2 \leq i \leq n-1$,
$$
b_i = 1 - c_i = V_i = \frac{U_{i,i+1}}{U_{i,i+1} + U_{i,i-1}}.
$$
where $(U_{i,i+1})_{i\geq1}$ are i.i.d.\  random variables of law $\mathcal{L}$
supported in $(0,\infty)$. The random variables $V_1,\ldots,V_n$ are dependent
here. 

Let us consider now an alternative simple way to make $K$ random. Namely, we
use a sequence $(V_i)_{i\geq1}$ of i.i.d.\ random variables on $[0,1]$ with
common law $\mathcal{L}$ and define the random birth-and-death Markov
kernel $K$ by \eqref{eq:chainKabc} with
$$
b_1=c_n=1 \quad\text{and}\quad
b_i=1-c_i=V_i
\quad\text{for every $2\leq i\leq n-1$.}
$$
In other words, the random Markov kernel $K$ is of the form
\begin{equation}\label{eq:chainK}
K=
\begin{pmatrix}
  0 & 1 \\
  1-V_2 & 0 & V_2 \\
  & 1-V_3 & 0 & V_3 \\
  && \ddots & \ddots & \ddots \\
  &&& 1-V_{n-1} & 0 & V_{n-1} \\
  &&&&   1 & 0
\end{pmatrix}.
\end{equation}
This is not a random conductance model. However, the kernel is a particular
case of the one appearing in Theorem \ref{th:chainerg}, corresponding to the
i.i.d.\ environment given by
$$
\mathbf{p}(i)=(1-V_i,0,V_i)
$$
for every $i\geq1$. This 
gives the following corollary of Theorem
\ref{th:chainerg}.

\begin{corollary}[Global behavior for i.i.d.\  environment]\label{co:chain}
  Let $K$ be the random birth-and-death Markov kernel \eqref{eq:chainK} where
  $(V_i)_{i\geq2}$ are i.i.d.\  of law $\mathcal{L}$ on $[0,1]$.
  Then there exists a non-random probability distribution $\mu$ on $[-1,+1]$
  such that almost surely,
  $$
  \lim_{n\to\infty} W_p(\mu_K,\mu)=0
  $$
  for every $p\geq1$, where $W_p$ is the Wasserstein distance as in
  \eqref{eq:Wp}. The limiting spectral distribution $\mu$ is fully
  characterized by its sequence of moments, given for every $k \geq 1$ by
  $$
  \int_{-1}^{+1}\!x^{2k-1}\,\mu(dx) = 0 %
  \quad\text{ and }\quad %
  \int_{-1}^{+1}\!x^{2k}\,\mu(dx) %
  = \sum_{\gamma \in D_k} %
  \prod_{i\in\mathds{Z}} %
  \mathds{E} \left(V ^{N_\gamma(i)}(1-V)^{N_\gamma(i-1)}\right)
  $$
  where $V$ is a random variable of law $\mathcal{L}$ and where
  $$
  D_k = \{ \gamma = (\gamma_0, \ldots,\gamma_{2k}) : \gamma_0 = \gamma_{2k} =
  0, \text{ and } |\gamma_\ell - \gamma_{\ell +1} | = 1 \text{ for every $0
    \leq \ell \leq 2k -1$}\}
  $$
  is the set of loop paths of length $2k$ of the simple random walk on
  $\mathds{Z}$, and
  $$
  N_\gamma (i) %
  = \sum_{\ell = 0}^{2 k-1} \mathds{1}_{\{(\gamma_\ell, \gamma_{\ell+1}) =
    (i,i+1)\}}
  $$
  is the number of times $\gamma$ crosses the horizontal line $y = i
  +\frac{1}{2}$ in the increasing direction.
\end{corollary}

When the random variables $(V_i)_{i\geq2}$ are only stationary and ergodic,
Corollary \ref{co:chain} remains valid provided that we adapt the
formula for the even moments of $\mu$ (that is, move the
product inside the expectation).

\begin{remark}[From Dirac masses to arc--sine laws]\label{re:arcsine}
  Corollary \ref{co:chain} gives a formula for the moments of $\mu$. This
  formula is a series involving the ``Beta-moments'' of $\mathcal{L}$. We
  cannot compute it explicitly for arbitrary laws $\mathcal{L}$ on $[0,1]$.
  However, in the deterministic case $\mathcal{L} = \delta_{1/2}$, we have,
  for every integer $k\geq1$,
  $$
  \int_{-1}^{+1}\! x^{2k} \mu (dx) %
  = \sum_{\gamma \in D_k} %
  2 ^{ - \sum_i  N_\gamma (i)- \sum_i  N_\gamma (i-1)} %
  =  2^{-2k} { \binom{2k}{k}} %
  = \int_{-1}^{+1}\!\! x^{2k} \frac{dx}{ \pi \sqrt{1 - x^2} }
  $$
  which confirms the known fact that $\mu$ is the arc--sine law on $[-1,+1]$
  in this case (see e.g.\ \cite[III.4 page 80]{MR0228020}). More generally, a
  very similar computation reveals that if $\mathcal{L}=\delta_p$ with $0<p<1$
  then $\mathcal{\mu}$ is the arc--sine law on
  $[-2\sqrt{p(1-p)}\,,\,+2\sqrt{p(1-p)}]$. Figures
  \ref{fi:tridiaga}-\ref{fi:tridiagb}-\ref{fi:tridiagc} display simulations
  illustrating Corollary \ref{co:chain} for various other choices of
  $\mathcal{L}$. 
\end{remark}

\begin{remark}[Non--universality] The law $\mu$ in Corollary \ref{co:chain} is
  not universal, in the sense that it depends on many ``Beta-moments'' of
  $\mathcal{L}$, in contrast with the complete graph case where the limiting
  spectral distribution depends on $\mathcal{L}$ only via its first two
  moments.
\end{remark}

We now turn to the edge behavior of $\sigma(K)$ where $K$ is as in
\eqref{eq:chainK}. Since $K$ has period $2$, one has $\lambda_n(K)=-1$ and we
are interested in the behavior of $\lambda_2(K)=-\lambda_{n-1}(K)$ as $n$ goes
to infinity. Since the limiting spectral distribution $\mu$ is symmetric, the
convex hull of its support is of the form $[-\alpha_\mu ,+\alpha_\mu ]$ for
some $0 \leq \alpha_\mu \leq 1$. The following result gives 
information on $\alpha_\mu$. The reader may forge many conjectures in the same
spirit for the map $\mathcal{L}\mapsto \mu$ from the simulations given by
Figures \ref{fi:tridiaga}-\ref{fi:tridiagb}-\ref{fi:tridiagc}.

\begin{theorem}[Edge behavior for i.i.d.\ environment] \label{th:chainedge} Let
  $K$ be the random birth-and-death Markov kernel \eqref{eq:chainK} where
  $(V_i)_{i\geq2}$ are i.i.d.\ of law $\mathcal{L}$ on $[0,1]$. Let $\mu$ be
  the symmetric limiting spectral distribution on $[-1,+1]$ which appears in
  Corollary \ref{co:chain}. Let $[-\alpha_\mu,+\alpha_\mu]$ be the convex hull
  of the support of $\mu$. If $\mathcal{L}$ has a positive density at $1/2$
  then $\alpha_\mu = 1$. Consequently, almost surely,
  $$
  \lambda_2(K)=-\lambda_{n-1}(K)=1+o(1).
  $$  
  On the other hand, 
  if $\mathcal{L}$ is supported on $[0,t]$ with $0<t<1/2$ or on
  $[t,1]$ with $1/2<t<1$ then almost surely
  $\limsup_{n\to\infty}\lambda_2(K)<1$ and therefore $\alpha_\mu<1$.
\end{theorem}

The proof of Theorem \ref{th:chainedge} is given in Section \ref{se:chain}.
The speed of convergence of $\lambda_2(K)-1$ to $0$ is highly dependent on the
choice of the law $\mathcal{L}$. As an example, if e.g.\
$$
\mathds{E}\left[\log\frac{V}{1-V}\right]=0 %
\quad\text{and}\quad %
\mathds{E}\left[\left(\log\frac{V}{1-V}\right)^2\right]>0
$$
where $V$ has law $\mathcal{L}$, then $K$ is the so called Sinai random walk on
$\{1,\ldots,n\}$. In this case, by a slight
modification of the analysis of \cite{MR2370603}, one can prove that almost
surely,
$$
-\infty<\liminf_{n\to\infty}\frac{1}{\sqrt{n}}\log(1-\lambda_2(K)) %
\leq \limsup_{n\to\infty}\frac{1}{\sqrt{n}}\log(1-\lambda_2(K))<0.
$$
Thus, the convergence to the edge here occurs exponentially fast in
$\sqrt{n}$. On the other hand, if for instance $\mathcal{L}=\delta_{1/2}$
(simple reflected random walk on $\{1,\ldots,n\}$) then it is known that
$1-\lambda_2(K)$ decays as $n^{-2}$ only. 

\medskip\bigskip

We conclude with a list of remarks and open problems.

\medskip

\noindent {\bf Fluctuations at the edge}. An interesting problem concerns the
fluctuations of $\lambda_2(\sqrt{n}K)$ around its limiting value $2\sigma$ in
the complete graph model. Under suitable moments conditions on $\mathcal{L}$,
one may seek for a deterministic sequence $(a_n)$, and a probability
distribution $\mathcal{D}$ on $\mathds{R}$ such that
\begin{equation}\label{fluttua}
a_n\left(
  \lambda_2(\sqrt{n}K)-2\sigma\right)
\overset{\text{d}}{\underset{n\to\infty}{\longrightarrow}}\mathcal{D}
\end{equation}
where ``$\overset{d}{\to}$'' stands for convergence in distribution. The same
may be asked for the random variable $\lambda_n(\sqrt{n}K) + 2\sigma$.
Computer simulations suggest that $a_n\approx n^{2/3}$ and that $\mathcal{D}$
is close to a Tracy-Widom distribution. The heuristics here is that
$\lambda_2(\sqrt{n}K)$ behaves like the $\lambda_1$ of a centered Gaussian
random symmetric matrix. The difficulty is that the entries of $K$ are not
i.i.d., not centered, and of course not Gaussian.

\medskip

\noindent {\bf Symmetric Markov generators}. Rather than considering the
random walk with infinitesimal generator $K-I$ on the complete graph as we
did, one may start with the \emph{symmetric} infinitesimal generator $G$
defined by $G_{i,j} = G_{j,i} = U_{i,j}$ for every $1\leq i<j\leq n$ and
$G_{i,i}=-\sum_{j\neq i}G_{i,j}$ for every $1\leq i\leq n$. Here
$(U_{i,j})_{1\leq i<j}$ is a triangular array of i.i.d.\ real random variables
of law $\mathcal{L}$. For this model, the uniform probability measure
$\mathcal{U}$ is reversible and invariant. The bulk behavior of such random
matrices has been investigated in \cite{MR2206341}.

\medskip

\noindent {\bf Non--reversible Markov ensembles}. A non--reversible model is
obtained when the underlying complete graph is oriented. That is each vertex
$i$ has now (besides the loop) $n-1$ outgoing edges $(i,j)$ and $n-1$ incoming
edges $(j,i)$. On each of these edges we place an independent positive weight
$V_{i,j}$ with law $\mathcal{L}$, and on each loop an independent positive
weight $V_{i,i}$ with law $\mathcal{L}$. This gives us a non--reversible
stochastic matrix
$$
\wt K_{i,j} = \frac{V_{i,j}}{\sum_{k=1}^n V_{i,k}}\,.
$$
The spectrum of $\wt K$ is now complex. If $\mathcal{L}$ is exponential, then
the matrix $\wt K$ describes the Dirichlet Markov Ensemble considered in
\cite{chafai-dme}. Numerical simulations suggest that if $\mathcal{L}$ has,
say, finite positive variance, then the ESD of $n^{1/2}\wt K$ converges weakly
as $n\to\infty$ to the uniform law on the unit disc of the complex plane
(circular law). At the time of writing, this conjecture is still open. Note
that the ESD of the i.i.d.\ matrix $(n^{-1/2}V_{i,j})_{1\leq {i,j} \leq n}$ is
known to converge weakly to the circular law; see \cite{tao-vu-cirlaw-bis} and
references therein.

\medskip

\noindent {\bf Heavy--tailed weights}. Recently, remarkable work has been
devoted to the spectral analysis of large dimensional symmetric random
matrices with heavy--tailed i.i.d.\ entries, see e.g.\
\cite{MR2081462,MR2548495,benarous-guionnet,Zakharevich,MR2371333}.
Similarly, on the complete graph, one may consider the bulk and edge behavior
of the random reversible Markov kernels constructed by \eqref{eq:srw} when the
law $\mathcal{L}$ of the weights is heavy--tailed (i.e.\ with at least an
infinite second moment). In that case, and in contrast with Theorem
\ref{th:wigner}, the scaling is not $\sqrt{n}$ and the limiting spectral
distribution is not Wigner's semi--circle law. We study such heavy--tailed
models elsewhere \cite{bordenave-caputo-chafai-ii}. Another interesting model
is the so called \emph{trap model} which corresponds to put heavy--tailed
weights only on the diagonal of $\mathbf{U}$ (holding times), see e.g.\
\cite{MR2152251} for some recent advances.

\section{Proofs for the complete graph model}
\label{se:wigner}

Here we prove Theorems \ref{th:wigner}, \ref{th:wigneredge}, Proposition
\ref{th:wignerinv} and Corollary \ref{co:wignerstrong}. In the whole sequel,
we denote by $L^2(1)$ the Hilbert space $\mathds{R}^n$ equipped with the
scalar product
$$
\scalar{x}{y} = \sum_{i=1}^nx_i\,y_i.
$$
The following simple lemma allows us to work with symmetric matrices when
needed.

\begin{lemma}[Spectral equivalence]\label{le:equiv}
  Almost surely, for large enough $n$, the spectrum of the reversible Markov
  matrix $K$ coincides with the spectrum of the symmetric matrix $S$ defined
  by
  $$
  S_{i,j}= \sqrt{\frac{\rho_i}{\rho_j}}K_{i,j}
  =\frac{U_{i,j}}{\sqrt{\rho_i\rho_j}}\,.
  $$
  Moreover, the corresponding eigenspaces dimensions also coincide.
\end{lemma}

\begin{proof}
  Almost surely, for large enough $n$, all the $\rho_i$ are positive and 
  $K$ is self--adjoint as an operator from $L^2(\rho)$ to $L^2(\rho)$, where
  $L^2(\rho)$ denotes $\mathds{R}^n$ equipped with the scalar product
  $$
  \left<x,y\right>_{\rho}=\sum_{i=1}^n\rho_i\,x_i\,y_i.
  $$
  It suffices to observe that 
  a.s.\ for large enough $n$, 
  the map $x\mapsto
  \widehat x$ defined by
  $$
  \widehat{x}=\left(x_1\sqrt{\rho_1},\ldots,x_n\sqrt{\rho_n}\right)
  $$
  is an isometry from $L^2(\rho)$ to $L^2(1)$ and that for any
  $x,y\in\mathds{R}^n$ and $1\leq i\leq n$, we have
  $$
  (Kx)_i=\sum_{j=1}^nK_{i,j}x_j
  $$
  and
  $$
  \left<Kx,y\right>_{\rho} %
  =\sum_{i,j=1}^n K_{i,j}x_jy_i\rho_i %
  =\sum_{i,j=1}^n U_{i,j}x_jy_i %
  =\sum_{i,j=1}^n S_{i,j}\widehat{x}_i\widehat{y}_j %
  =\left<S\widehat{x},\widehat{y}\right>.
  $$
\end{proof}

The random symmetric matrix $S$ has non--centered, non--independent entries.
Each entry of $S$ is bounded and belongs to the interval $[0,1]$, since for
every $1\leq i,j\leq n$, we have $S_{i,j}\leq
U_{i,j}/\sqrt{U_{i,j}U_{j,i}}=1$. In the sequel, for any $n\times n$ real
symmetric matrix $A$, we denote by
$$
\lambda_n(A)\leq\cdots\leq\lambda_1(A)
$$
its ordered spectrum. We shall also denote by $\|A\|$ the operator norm of
$A$, defined by
$$
\|A\|^2 = \max_{x\in\mathds{R}^n}\frac{\scalar{Ax}{A x}}{\scalar{x}{x}}.
$$
Clearly, $\|A\| = \max(\lambda_1(A),-\lambda_n(A))$. To prove Theorem
\ref{th:wigner} we shall compare the symmetric random matrix $\sqrt n\,S$ with
the symmetric $n\times n$ random matrices
\begin{equation}\label{eq:wi}
  W_{i,j} = \frac{U_{i,j} -1}{\sqrt n} %
  \quad\text{and}\quad %
  \wt W_{i,j} = \frac{U_{i,j}}{\sqrt n}.
\end{equation}
Note that $W$ defines a so called Wigner matrix, i.e.\ $W$ is symmetric and it
has centered i.i.d.\ entries with finite positive variance. We shall also need
the non--centered matrix $\wt W$. It is well known that under the sole
assumption $\sigma^2\in(0,\infty)$ on $\mathcal{L}$, almost surely,
$$
\mu_{W} %
\underset{n\to\infty}{\overset{w}{\longrightarrow}}\mathcal{W}_{2\sigma}%
\quad\text{and}\quad %
\mu_{\wt W} %
\underset{n\to\infty}{\overset{w}{\longrightarrow}}\mathcal{W}_{2\sigma}
$$
where $\mu_W$ and $\mu_{\wt W}$ are the ESD of $W$ and $\wt W$, see e.g.\ 
\cite[Theorems 2.1 and 2.12]{MR1711663}. Note that $\wt W$ is a rank one
perturbation of $W$, which implies that the spectra of $W$ and $\wt W$ are
interlaced (Weyl-Poincar\'e inequalities, see e.g.\ 
\cite{MR1091716,MR1711663}). Moreover, under the assumption of finite fourth
moment on $\mathcal{L}$, it is known that almost surely
$$
\lambda_n(W)\to -2\sigma  \quad\text{and}\quad \lambda_1(W)\to +2\sigma.
$$
In particular, almost surely,
\begin{equation}\label{eq:wio} 
  \|W\| = 2\sigma + o(1)\,.
\end{equation}
On the other hand, and still under the finite fourth moment assumption, almost
surely,
$$
\lambda_1(\wt W)\to+\infty %
\quad\text{while}\quad %
\lambda_2(\wt W)\to +2\sigma %
\quad\text{and}\quad %
\lambda_{n}(\wt W)\to -2\sigma
$$
see e.g.\ \cite{MR958213,MR637828,MR1711663}. Heuristically, when $n$ is
large, the law of large numbers implies that $\rho_i$ is close to $n$ (recall
that here $\mathcal{L}$ has mean $1$), and thus $\sqrt n \,S$ is close to $\wt
W$. The main tools needed for a comparison of the matrix $\sqrt{n}S$ with $\wt
W$ are given in the following subsection.

\subsection*{Uniform law of large numbers}

We shall need the following Kolmogorov-Marcinkiewicz-Zygmund strong uniform
law of large numbers, related to Baum-Katz type theorems.

\begin{lemma}
\label{le:LLN-gen}
  Let $(A_{i,j})_{i,j\geq1}$ be a symmetric array of i.i.d.\  random variables.
  For any reals $a>1/2$, $b\geq0$, and $M>0$, if
  $\mathds{E}(|A_{1,1}|^{(1+b)/a})<\infty$ then
  $$
  \max_{1\leq i\leq Mn^b}\biggr|\sum_{j=1}^n(A_{i,j}-c)\biggr|=o(n^a) %
  \quad\text{a.s.}\quad\text{where}\quad %
  c=\begin{cases}
    \mathds{E}(A_{1,1}) & \text{if $a\leq1$} \\
    \text{any number} & \text{if $a>1$}.
  \end{cases}
  $$
\end{lemma}

\begin{proof}
  This result is proved in \cite[Lemma 2]{MR1235416} for a non--symmetric
  array. The symmetry makes the random variables
  $(\sum_{j=1}^nA_{i,j})_{i\geq1}$ dependent, but a careful analysis of the
  argument shows that this is not a problem except for a sort of converse, see
  \cite[Lemma 2]{MR1235416} for details.
\end{proof}

\begin{lemma}
\label{le:LLN}
  If $\mathcal{L}$ has finite moment of order $\kappa\in[1,2]$ then 
  \begin{equation}\label{eq:LLN0}
    \max_{1\leq i\leq n^{\kappa-1}}\left|\frac{\rho_i}{n}-1\right|=o(1)
  \end{equation}
  almost surely, and in particular, if $\mathcal{L}$ has finite second moment,
  then almost surely
  \begin{equation}\label{eq:LLN1}
    \max_{1\leq i\leq n}\left|\frac{\rho_i}{n}-1\right|=o(1).
  \end{equation}
  Moreover if $\mathcal{L}$ has finite moment of order $\kappa$ with
  $2\leq\kappa<4$, then almost surely
  \begin{equation}\label{eq:LLN2}
  \max_{1\leq i\leq n}\left|\frac{\rho_i}{n}-1\right|=o(n^{\frac{2-\kappa}{\kappa}}).
  \end{equation}  
  Additionally, if $\mathcal{L}$ has finite fourth moment, then almost surely
  \begin{equation}\label{eq:LLN3}
  \sum_{i=1}^n\left(\frac{\rho_i}{n}-1\right)^2 =O(1)\,. %
  \end{equation}
\end{lemma}

\begin{proof} 
  The result \eqref{eq:LLN0} follows from Lemma \ref{le:LLN-gen} with
  $$
  A_{i,j}=U_{i,j},\quad a=M=1,\quad b=\kappa-1.
  $$
  We recover the standard strong law of large numbers with $\kappa=1$. The
  result \eqref{eq:LLN2} -- and therefore \eqref{eq:LLN1} setting
  $\kappa=2$ --  follows from Lemma \ref{le:LLN-gen} with this time
  $$
  A_{i,j}=U_{i,j},\quad a=2/\kappa,\quad b=M=1.
  $$

  Proof of \eqref{eq:LLN3}. We set $\epsilon_i=n^{-1}\rho_i-1$ for every
  $1\leq i\leq n$. Since $\mathcal{L}$ has finite fourth moment, the result
  \eqref{eq:wio} for the centered Wigner matrix $W$ defined by \eqref{eq:wi}
  gives that 
  $$
  \sum_{i=1}^n\epsilon_i^2 %
  =\frac{\scalar{W1}{W1}}{\scalar{1}{1}} %
  \leq \|W\|^2 %
  =4\sigma^2+o(1)=O(1)
  $$
  almost surely.
\end{proof}

We are now able to give a proof of Proposition \ref{th:wignerinv}.

\begin{proof}[Proof of Proposition \ref{th:wignerinv}]
  Since $\mathcal{L}$ has finite first moment, by the strong law of large
  numbers,
  $$
  \rho_1+\cdots+\rho_n %
  =\sum_{i=1}^nU_{i,i} %
  +2\!\!\sum_{1\leq i<j\leq n}\!\! U_{i,j} %
  =n^2(1+o(1))
  $$
  almost surely. For every fixed $i\geq1$, we have also $\rho_i=n(1+o(1))$
  almost surely. As a consequence, for every fixed $i\geq1$, almost surely,
  $$
  \hat{\rho}_i %
  =\frac{\rho_i}{\rho_1+\cdots+\rho_n} %
  =\frac{n(1+o(1))}{n^2(1+o(1))} %
  =\frac{1}{n}(1+o(1)).
  $$
  Moreover, since $\mathcal{L}$ has finite second moment, the $o(1)$ in the
  right hand side above is uniform over $1\leq i\leq n$ thanks to
  \eqref{eq:LLN1} of Lemma \ref{le:LLN}. This achieves the proof.
\end{proof}

Note that, under the second moment assumption, 
$\hat{\rho}_i=n^{-1}(1+O(\delta))$ for $1\leq i\leq n$, where
\begin{equation}\label{eq:delta}
  \delta: = \max_{1\leq i\leq n} |\epsilon_i|=o(1)\,, %
  \quad\text{with}\quad %
  \epsilon_i:=n^{-1}\rho_i-1.
\end{equation}
We will repeatedly use the notation (\ref{eq:delta}) in the sequel.

\subsection*{Bulk behavior}

Lemma \ref{le:equiv} reduces Theorem \ref{th:wigner} to the study of the ESD
of $\sqrt{n}S$, a symmetric matrix with non independent entries. One can find
in the literature many extensions of Wigner's theorem to symmetric matrices
with non--i.i.d.\ entries. However, none of these results seems to apply here
directly.

\begin{proof}[Proof of Theorem \ref{th:wigner}]
  We first recall a standard fact about comparison of spectral densities of
  symmetric matrices. Let $L(F,G)$ denote the L\'evy distance between two
  cumulative distribution functions $F$ and $G$ on $\mathds{R}$, defined by
  $$
  L(F,G)=\inf\{\varepsilon>0\text{ such that } %
  F(\cdot-\varepsilon)-\epsilon %
  \leq G \leq F(\cdot+\epsilon)+\epsilon)\}\,.
  $$
  It is well known \cite{MR1700749} that the L\'evy distance is a metric for
  weak convergence of probability distributions on $\mathds{R}$. If $F_A$ and
  $F_B$ are the cumulative distribution functions of the empirical spectral
  distributions of two hermitian $n\times n$ matrices $A$ and $B$, we have the
  following bound for the third power of $L(F_A,F_B)$ in terms of the trace of
  $(A-B)^2$:
  \begin{equation}\label{eq:l3bound}
    L^3(F_A,F_B)\leq\frac{1}{n}\mathrm{Tr}((A-B)^2)
    =\frac{1}{n}\sum_{i,j=1}^n (A_{i,j}-B_{i,j})^2\,.
  \end{equation}
  The proof of this estimate is a consequence of the Hoffman-Wielandt
  inequality \cite{MR0052379}, see also \cite[Lemma 2.3]{MR1711663}. 
  By Lemma \ref{le:equiv}, we have $\sqrt{n}\lambda_k(K)=\lambda_k(\sqrt{n}S)$
  for every $1\leq k\leq n$. We shall use the bound \eqref{eq:l3bound} for the
  matrices $A=\sqrt{n}S$ and $B=\wt W$, where $\wt W$ is defined in
  \eqref{eq:wi}. We will show that a.s.
  \begin{equation}\label{eq:clab}
    \frac{1}{n}\sum_{i,j=1}(A_{i,j}-B_{i,j})^2 %
    =O(\delta^2)\,,
  \end{equation}
  where $\delta=\max_i |\epsilon_i|$ as in \eqref{eq:delta}. Since
  $\mathcal{L}$ has finite positive variance, we know that the ESD of $B$
  tends weakly as $n\to\infty$ to the semi--circle law on
  $[-2\sigma,+2\sigma]$. Therefore the bound \eqref{eq:clab}, with
  \eqref{eq:l3bound} and the fact that $\delta\to 0$ as $n\to\infty$ is
  sufficient to prove the theorem. We turn to a proof of \eqref{eq:clab}. For
  every $1\leq i,j\leq n$, we have
  $$
  A_{i,j}-B_{i,j}
  =\frac{U_{i,j}}{\sqrt{n}}\left(\frac{n}{\sqrt{\rho_i\rho_j}}-1\right)\,.
  $$
  Set, as usual $\rho_i = n(1+\epsilon_i)$ and define $\psi_i =
  (1+\epsilon_i)^{-\frac12} - 1$. Note that by Lemma \ref{le:LLN}, almost
  surely, $\psi_i = O(\delta)$ uniformly in $i=1,\dots,n$. Also,
  $$
  \frac{n}{\sqrt{\rho_i\rho_j}}-1 = (1+\psi_i)(1+\psi_j) -1 =
  \psi_i+\psi_j + \psi_i\psi_j\,.
  $$
  In particular, $\frac{n}{\sqrt{\rho_i\rho_j}}-1 = O(\delta)$.
  Therefore
  $$ \frac{1}{n}\sum_{i,j=1}(A_{i,j}-B_{i,j})^2
  \leq O(\delta^2)\, \left(\frac{1}{n^2}\sum_{i,j=1}^n U_{i,j}^2\right)\,.
  $$ 
  By the strong law of large numbers, $\frac{1}{n^2}\sum_{i,j=1}^n U_{i,j}^2
  \to \sigma^2 + 1$ a.s., which implies \eqref{eq:clab}.
\end{proof}

\subsection*{Edge behavior}
We turn to the proof of Theorem \ref{th:wigneredge} which concerns the
edge of $\sigma(\sqrt{n}S)$.

\begin{proof}[Proof of Theorem \ref{th:wigneredge}]
  Thanks to Lemma \ref{le:equiv} and the global behavior proven in Theorem
  \ref{th:wigner}, it is enough to show that, almost surely,
  $$
  \limsup_{n\to\infty}\sqrt{n}\max(|\lambda_2(S)|,|\lambda_n(S)|)
  \leq 2\sigma\,.
  $$
  Since $K$ is almost surely irreducible for large enough $n$, the eigenspace
  of $S$ of the eigenvalue $1$ is almost surely of dimension $1$, and is
  given by $\mathds{R}(\sqrt{\rho_1},\ldots,\sqrt{\rho_n})$. Let $P$ be the
  orthogonal projector on $\mathds{R}\sqrt{\rho}$. The matrix $P$ is $n\times
  n$ symmetric of rank $1$, and for every $1\leq i,j\leq n$,
  $$
  P_{i,j}=\frac{\sqrt{\rho_i\rho_j}}{\sum_{k=1}^n\rho_k}.
  $$
  The spectrum of the symmetric matrix $S-P$ is
  $$
  \{\lambda_n(S),\ldots,\lambda_2(S)\}\cup\{0\}.
  $$
  By subtracting $P$ from $S$ we remove the largest eigenvalue $1$ from the
  spectrum, without touching the remaining eigenvalues. Let $\mathcal{V}$ be
  the random set of vectors of unit Euclidean norm 
  which are orthogonal to $\sqrt{\rho}$ for the scalar product
  $\left<\cdot,\cdot\right>$ of $\mathds{R}^n$. We have then
  $$
  \sqrt{n}\max(|\lambda_2(S)|,|\lambda_n(S)|)
  =\max_{v\in\mathcal{V}} \left|\left<\sqrt{n}Sv,v\right>\right|
  =\max_{v\in\mathcal{V}} |\scalar{\wt Av}{v}|
  $$
  where $\wt A$ is the $n\times n$ random symmetric matrix defined by
  $$
  \wt A_{i,j}=\sqrt{n}(S-P)_{i,j}
  =\sqrt{n}\left(\frac{U_{i,j}}{\sqrt{\rho_i\rho_j}}
    -\frac{\sqrt{\rho_i\rho_j}}{\sum_{k=1}^n\rho_k}\right).
  $$
  In Lemma \ref{le:control} below we establish that almost surely $
  \scalar{v}{(\wt A - W)v} = O(\delta)+O(n^{-1/2})$ uniformly in
  $v\in\mathcal{V}$, where $W$ is defined in \eqref{eq:wi} and $\delta$ is
  given by \eqref{eq:delta}. Thus, using \eqref{eq:wio},
  $$
  \left|\scalar{Wv}{v}\right| \leq \max (|\lambda_1(W)|,|\lambda_n(W)|) = 2\sigma + o(1)\,,
  $$
  we obtain that almost surely, uniformly in $v\in\mathcal{V}$,
  $$
  | \scalar{\wt Av}{v}|%
  \leq |\scalar{Wv}{v}| + |\scalar{(\wt A- W)v}{v}| %
  = 2\sigma+o(1)+O(\delta).%
  $$
  Thanks to Lemma \ref{le:LLN} we know that $\delta=o(1)$ and the theorem
  follows.
\end{proof}

\begin{lemma}\label{le:control}
  Almost surely, uniformly in $v\in\mathcal{V}$,
  we have, with $\delta:=\max_i |\epsilon_i|$,
  $$
  \scalar{v}{(\wt A - W)v} = O(\delta) + O(n^{-1/2}).
  $$
\end{lemma}

\begin{proof}
  We start by rewriting the matrix
  $$
  \wt A_{i,j}=\frac{\sqrt n \,U_{i,j}}{\sqrt{\rho_i\rho_j}} - \frac{\sqrt n
    \sqrt{\rho_i\rho_j}}{\sum_k\rho_k}
  $$
  by expanding around the law of large numbers. We set $\rho_i =
  n(1+\epsilon_i)$ and we define
  $$
  \varphi_i = \sqrt{1+\epsilon_i}\, -\, 1 %
  \quad\text{and}\quad %
  \psi_i= \frac1{\sqrt{1+\epsilon_i}} \,-\,1.
  $$
  Observe that $\varphi_i$ and $\psi_i$ are of order $\epsilon_i$ and by Lemma
  \ref{le:LLN}, cf.\ \eqref{eq:LLN3} we have a.s.
  \begin{equation}\label{eq:01}
    \scalar{\varphi}{\varphi} = \sum_i \varphi_i^2 = O(1)
    \quad\text{and}\quad
    \scalar{\psi}{\psi} = \sum_i \psi_i^2 = O(1)\,.
  \end{equation} 
  We expand 
  $$
  \sqrt{\rho_i\rho_j} %
  = n(1+\epsilon_i)^\frac12(1+\epsilon_j)^\frac12 %
  = n(1+\varphi_i)(1+\varphi_j)\,.
  $$
  Similarly, we have
  $$
  \frac{1}{\sqrt{\rho_i\rho_j}} = n^{-1}(1+\psi_i)(1+\psi_j).
  $$
  Moreover, writing 
  $$
  \sum_{k=1}^n \rho_k = n^2\left(1 + \frac1n\sum_k\epsilon_k\right)
  $$
  and setting $\gamma:= (1+\frac1n\sum_k\epsilon_k)^{-1} - 1 $ we see that
  $$
  \left(\sum_{k=1}^n \rho_k \right)^{-1}=
  \frac1{n^2}\,(1 +\gamma) \,.
  $$
  Note that $\gamma=O(\delta)$.
  Using these expansions we obtain 
  $$
  \frac{\sqrt n \,U_{i,j}}{\sqrt{\rho_i\rho_j}}  = \frac1{\sqrt n}\,
  U_{i,j}(1+\psi_i)(1+\psi_j)
  $$
  and
  $$
  \frac{\sqrt n \sqrt{\rho_i\rho_j}}{\sum_k\rho_k}
  = \frac1{\sqrt n}
  \,(1 +\varphi_i)(1+\varphi_j)(1+\gamma)\,.
  $$
  From these expressions, with the definitions $$
  \Phi_{i,j} = \varphi_i + \varphi_j + \varphi_i\varphi_j %
  \quad\text{and}\quad %
  \Psi_{i,j} = \psi_i + \psi_j + \psi_i\psi_j,
  $$
  we obtain
  $$
  \wt A_{i,j} = W_{i,j}(1 + \Psi_{i,j}) %
  +  \frac1{\sqrt n} \, \left[\Psi_{i,j} %
    - \Phi_{i,j}(1+\gamma) + \gamma\right].
  $$
  Therefore, we have
  $$
  \scalar{v}{(W-\wt A)v} %
  = - \sum_{i,j} v_i W_{i,j} \Psi_{i,j} v_j\; %
  +\frac{1+\gamma}{\sqrt n} \,\scalar{v}{\Phi v} %
  -\frac1{\sqrt n} \,\scalar{v}{\Psi v} %
  -\frac\gamma{\sqrt n} \,\scalar{v}{1}^2.
  $$
  Let us first show that 
  \begin{equation}\label{eq:2}
    \scalar{v}{1} = O(1)\,.
  \end{equation}
  Indeed, $v\in\mathcal{V}$ implies that for any $c\in\mathds{R}$,
  $$
  \scalar{v}{1} = \scalar{v}{1-c\sqrt{\rho}}.
  $$
  Taking $c=1/\sqrt n$ we see that
  $$
  1-c\sqrt{\rho_i} %
  = 1 - \sqrt{1+\epsilon_i} = -\varphi_i.
  $$ 
  Thus, Cauchy--Schwarz' inequality implies 
  $$
  \scalar{v}{1}^2\leq
  \scalar{v}{v}\scalar{\varphi}{\varphi}
  $$
  and \eqref{eq:2} follows from \eqref{eq:01} above. Next, we show that
  \begin{equation}\label{eq:3}
    \scalar{v}{\Phi v} = O(1).
  \end{equation}
  Note that
  $$
  \scalar{v}{\Phi v} %
  = 2\scalar{v}{1}\scalar{v}{\varphi} %
  +\scalar{v}{\varphi}^2.
  $$
  Since $\scalar{v}{\varphi}^2 \leq\scalar{v}{v}\scalar{\varphi}{\varphi}$ we
  see that \eqref{eq:3} follows from \eqref{eq:01} and \eqref{eq:2}. In the
  same way we obtain that $\scalar{v}{\Psi v} = O(1)$. So far we have obtained
  the estimate
  \begin{equation}\label{eq:sfa}
    \scalar{v}{(W-\wt A)v} %
    = - \sum_{i,j} v_i W_{i,j} \Psi_{i,j} v_j + O(n^{-1/2}).
  \end{equation}
  To bound the first term above we observe that 
  \begin{align*}
    \sum_{i,j} v_i W_{i,j} \Psi_{i,j} v_j&  = 2 \sum_i \psi_i v_i (W  v)_i
    + \sum_{i,j} \psi_i v_i  W_{i,j} \psi_j v_j\\ &=
    2\scalar{\hat \psi}{Wv} + \scalar{\hat\psi}{W\hat\psi}\,,  
  \end{align*}
  where $\hat \psi$ denotes the vector $\hat \psi_i :=\psi_i v_i$. 
  Note that 
  $$
  \scalar{\hat\psi}{\hat\psi} %
  = \sum_i\psi_i^2 v_i^2 %
  \leq O(\delta^2)\scalar{v}{v} %
  = O(\delta^2).
  $$
  Therefore, by definition of the norm $\|W\|$
  $$
  |\scalar{\hat\psi}{W\hat\psi}|\leq 
  \sqrt{\scalar{\hat\psi}{\hat\psi}}
  \sqrt{\scalar{W\hat\psi}{W\hat\psi}} \leq 
  \|W\|\,\scalar{\hat\psi}{\hat\psi}
  = O(\delta^2)\,\|W\|\,.
  $$
  Similarly, we have
  $$
  |\scalar{\hat \psi}{Wv}|%
  \leq \sqrt{\scalar{\hat\psi}{\hat\psi}}
  \sqrt{\scalar{Wv}{Wv}} %
  \leq O(\delta)\,\|W\| \sqrt{\scalar{v}{v}}
  = O(\delta)\,\|W\|\,.
  $$
  From \eqref{eq:wio}, $\|W\| = 2\sigma + o(1) = O(1)$. Therefore, going back to
  \eqref{eq:sfa} we have obtained
  $$
  \scalar{v}{(W-\wt A)v} = O(\delta) + O(n^{-1/2}).
  $$
\end{proof} 

We end this section with the proof of Corollary \ref{co:wignerstrong}.

\begin{proof}[Proof of Corollary \ref{co:wignerstrong}]  
  By Theorem \ref{th:wigneredge}, almost surely, and for any compact subset
  $C$ of $\mathds{R}$ containing strictly $[0,2\sigma]$, the law
  $\wt\mu_{\sqrt{n}K}$ is supported in $C$ for large enough $n$. On the other
  hand, since
  $\mu_{\sqrt{n}K}=(1-n^{-1})\wt\mu_{\sqrt{n}K}+n^{-1}\delta_{\sqrt{n}}$, we
  get from Theorem \ref{th:wigner} that almost surely, $\wt\mu_{\sqrt{n}K}$
  tends weakly to $\mathcal{W}_{2\sigma}$ as $n\to\infty$. Now, for sequences
  of probability measures supported in a common compact set, by 
  Weierstrass' theorem, weak convergence is equivalent to Wasserstein
  convergence $W_p$ for every $p\geq1$. Consequently, almost surely,
  \begin{equation}\label{eq:cvmut}
    \lim_{n\to\infty}W_p(\wt\mu_{\sqrt{n}K},\mathcal{W}_{2\sigma})=0.
  \end{equation}
  for every $p\geq1$. It remains to study
  $W_p(\mu_{\sqrt{n}K},\mathcal{W}_{2\sigma})$. Recall that if $\nu_1$ and
  $\nu_2$ are two probability measures on $\mathds{R}$ with cumulative
  distribution functions $F_{\nu_1}$ and $F_{\nu_2}$ with respective
  generalized inverses $F_{\nu_1}^{-1}$ and $F_{\nu_2}^{-1}$, then, for every
  real $p\geq1$, we have, according to e.g.\ \cite[Remark 2.19
  (ii)]{MR1964483},
  \begin{equation}\label{eq:wpmut}
    W_p(\nu_1,\nu_2)^p %
    = \int_0^1\!%
    \left|F_{\nu_1}^{-1}(t)-F_{\nu_2}^{-1}(t)\right|^p\,dt.
  \end{equation}
  Let us take
  $\nu_1=\mu_{\sqrt{n}K}=(1-n^{-1})\wt\mu_{\sqrt{n}K}+n^{-1}\delta_{\sqrt{n}}$
  and $\nu_2=\mathcal{W}_{2\sigma}$. Theorem \ref{th:wigneredge} gives
  $\lambda_2(\sqrt{n}K)<\infty$ a.s. Also, a.s., for large enough $n$, and for
  every $t\in(0,1)$,
  $$
  F_{\nu_1}^{-1}(t)=
  F_{\mu_{\sqrt{n}K}}^{-1}(t)=
  \sqrt{n}\mathds{1}_{[1-n^{-1},1)}(t)
  +F_{\wt\mu_{\sqrt{n}K}}^{-1}(t+n^{-1})\mathds{1}_{(0,1-n^{-1})}(t).
  $$
  The desired result follows then by plugging this identity in
  \eqref{eq:wpmut} and by using \eqref{eq:cvmut}.
\end{proof}

\section{Proofs for the chain graph model}
\label{se:chain}
In this section we prove the bulk results in Theorem \ref{th:chainerg} and
Corollary \ref{co:chain} and the edge results in Theorem \ref{th:chainedge}.

\subsection*{Bulk behavior}
\begin{proof}[Proof of Theorem \ref{th:chainerg}]
  Since $\mu_K$ is supported in the compact set $[-1,+1]$ which does not
  depend on $n$, Weierstrass' theorem implies that the weak convergence of
  $\mu_K$ as $n\to\infty$ is equivalent to the convergence of all moments, and
  is also equivalent to the convergence in Wasserstein distance $W_p$ for
  every $p\geq1$. Thus, it suffices to show that a.s. for any $\ell\geq0$, the
  $\ell^\text{th}$ moment of $\mu_K$ converges to
  $\mathds{E}[r_\ell^\mathbf{p}(0)]$ as $n\to\infty$. The sequence
  $(\mathds{E}[r^{\mathbf{p}}_\ell(0)])_{\ell\geq0}$ will be then necessarily
  the sequence of moments of a probability measure $\mu$ on $[-1,+1]$ which is
  the unique adherence value of $\mu_K$ as $n\to\infty$.

  For any $\ell\geq0$ and $i\geq1$ let $r_\ell^{\mathbf{p},n}(i)$ be the
  probability of return to $i$ after $\ell$ steps for the random walk on
  $\{1,\ldots,n\}$ with kernel $K$. Clearly,
  $r_\ell^{\mathbf{p},n}(i)=r_\ell^{\mathbf{p}}(i)$ whenever
  $1+\ell<i<n-\ell$. Therefore, for every fixed $\ell$, the ergodic theorem
  implies that almost surely,
  $$
  \lim_{n\to\infty}\frac{1}{n}\sum_{i=1}^nr_\ell^{\mathbf{p},n}(i) %
  =\lim_{n\to\infty}\frac{1}{n}\sum_{i=1}^nr_\ell^{\mathbf{p}}(i) %
  =\mathds{E}[r_\ell^{\mathbf{p}(0)}].
  $$
  This ends the proof.
\end{proof}

\begin{proof}[Proof of Corollary \ref{co:chain}]
  The desired convergence follows immediately from Theorem \ref{th:chainerg}
  with $\mathbf{p}(i)=(1-V_i,0,V_i)$ for every $i\geq1$. The expression of the
  moments of $\mu$ follows from a straightforward path--counting
  argument for the return probabilities of a one-dimensional random
  walk. \end{proof}  

Let us mention that the proof of Corollary \ref{co:chain} could have been
obtained via the trace-moment method for symmetric tridiagonal matrices.
Indeed, an analog of Lemma \ref{le:equiv} allows one to replace $K$ by a
symmetric tridiagonal matrix $S$. Although the entries of $S$ are not
independent, the desired result follows from a variant of the proof used by
Popescu for symmetric tridiagonal matrices with independent entries
\cite[Theorem 2.8]{MR2480789}. We omit the details.

\begin{remark}[Computation of the moments of $\mu$ for Beta environments]
  \label{re:beta}
  As noticed in Remark \ref{re:arcsine}, the limiting spectral distribution
  $\mu$ is the arc--sine law when $\mathcal{L}=\delta_{1/2}$. Assume now that
  $\mathcal{L}$ is uniform on $[0,1]$. Then for every integers $m\geq0$ and
  $n\geq0$,
  $$
  \mathds{E}(V^{m}(1 - V)^{n}) %
  = \int_0^1\!u^m(1-u)^n\,du %
  = \mathrm{Beta}(n+1,m+1) 
  = \frac{\Gamma(n+1)\Gamma(m+1)}{\Gamma(n+m+2)}
  $$
  which gives
  $$
  \mathds{E}(V^{m} (1 - V)^{n}) %
  = \frac{n!m!}{(n+m+1)!} %
  = \frac{1}{(n+m+1)\binom{n+m}{m}}. 
  $$
  The law of $\binom{n+m}{m}V^m(1-V)^n$ is the law of the probability of
  having $m$ success in $n+m$ tosses of a coin with a probability of success
  $p$ uniformly distributed in $[0,1]$. Similar formulas may be obtained when
  $\mathcal{L}$ is a Beta law $\mathrm{Beta}(\alpha,\beta)$.
\end{remark}

\subsection*{Edge behavior}
\begin{proof}[Proof of Theorem \ref{th:chainedge}] 
  Proof of the first statement. It is enough to show that for every $0<a<1$,
  there exists an integer $k_a$ such that for all $k \geq k_a$,
  \begin{equation} \label{eq:mulowerbound} %
    \int_{-1}^{+1}\! x^{2k} \mu (dx) \geq a^{2k}.
  \end{equation}
  By assumption, there exists $C > 0$ and $0< t_0 <1/2$ such that for all
  $0<t<t_0$,
  $$
  \mathds{P}( V \in [1/2 - t , 1/2 + t] ) \geq C t
  $$
  where $V$ is random variable of law $\mathcal{L}$. In particular, for all $0
  < t < t_0$,
  $$
  \mathds{E}\left[ V ^{N_\gamma (i)} (1 - V)^{N_\gamma (i-1)}\right] %
  \geq C t \left( \frac1 2 - t \right) ^{ N_\gamma (i)+ N_\gamma (i-1)},
  $$
  and, if 
  $$
  \| \gamma \|_{\infty} %
  = \max \{ i \geq 0 : \max( N_\gamma(i), N_\gamma(-i) )  \geq 1 \}
  $$
  then
  \begin{align*}
    \int_{-1}^{+1}\! x^{2k} \mu (dx) 
    &\geq \sum_{\gamma \in D_k}  %
    \prod_{i \in\mathds{Z}} C t\left(\frac{1}{2}-t\right)^{N_\gamma(i)+N_\gamma(i-1)}\\
    &\geq \sum_{\gamma \in D_k}%
    ( C t ) ^{ 2 \|\gamma\|_{\infty} } %
    \left(\frac{1}{2} - t\right)^{ \sum_i  N_\gamma (i)+ N_\gamma (i-1)} \\
    & \geq\left(\frac{1}{2} - t\right)^{2k} %
    \sum_{\gamma \in D_k} ( C t ) ^{ 2 \|\gamma\|_{\infty} } \\
    & \geq\left(\frac{1}{2}-t\right)^{2k}|D_{k,\alpha}|(C t)^{ 2 k^\alpha }, 
  \end{align*}
  where $D_{k,\alpha} = \{ \gamma \in D_k : \| \gamma \|_{\infty} \leq
  k^{\alpha} \}$. Now, from the Brownian Bridge version of Donsker's Theorem
  (see e.g.\ \cite{MR2386089} and references therein),
  for all $\alpha > 1/2$,
  $$
  \lim_{k \to \infty} \frac{ | D_{k,\alpha} | }{ | D_k |} = 1.
  $$
  Since $|D_k|=\mathrm{Card}(D_k)=\binom{2k}{k}$, Stirling's formula gives $|
  D_k | \sim 4^{k} (\pi k )^{-1/2}$, and thus
  $$
  \int_{-1}^{+1}\! x^{2k} \mu (dx) %
  \geq  (\pi k )^{-1/2} ( 1  - 2 t )^{2k}  ( C t ) ^{ 2 k^\alpha } (1 + o(1)).
  $$
  We then deduce the desired result \eqref{eq:mulowerbound} by taking $t$
  small enough such that $1 - 2t > a$ and $1/2 < \alpha < 1$. This achieves
  the proof of the first statement.

  Proof of the second statement. One can observe that if
  $\mathcal{L}=\delta_p$ for some $p\in(0,1)$ with $p\neq1/2$, an explicit
  computation of the spectrum will provide the desired result, in accordance
  with Remark \ref{re:arcsine}. For the general case, we get from
  \cite{MR1710983}, for any $2\leq k\leq n-1$,
  $$
  1-\lambda_2(K) \geq \frac{1}{4\max(B_k^+,B_k^-)}
  $$
  where
  $$
  B_k^+=\max_{i> k}\left[\left(\sum_{j=k+1}^i\frac{1}{\rho_j(1-V_j)}\right)
    \sum_{j\geq i}\rho_j\right]
  \quad\text{and}\quad
  B_k^-=\max_{i<k}\left[\left(\sum_{j=i}^{k-1}\frac{1}{\rho_jV_j}\right)
  \sum_{j\leq i}\rho_j\right]
  $$
  with the convention $V_1=1-V_n=1$. Here we have fixed the value of $n$ and
  $\rho$ is any invariant (reversible) measure for $K$. It is convenient to
  take $\rho_1=1$ and for every $2\leq i\leq n$
  $$
  \rho_i=\frac{V_2\cdots V_{i-1}}{(1-V_2)\cdots(1-V_i)}.
  $$
  By symmetry, it suffices to consider the case where $\mathcal{L}$ is
  supported in $[0,t]$ with $0<t<1/2$. Let us take $k=2$. In this case,
  $B_2^-=1$, and the desired result will follow if we show that $B_2^+$ is
  bounded above by a constant independent of $n$. To this end, we remark first
  that for any $\ell>j$ we have
  $\rho_\ell=\rho_j\prod_{m=j}^{\ell-1}(V_m/(1-V_{m+1}))$. Therefore, setting
  $e^{-\gamma}=t/(1-t)<1$, we have $\rho_\ell\leq\rho_j e^{-\gamma(\ell-j)}$.
  It follows that, for any $k<i$,
  \begin{align*}
    \sum_{j=k+1}^i\sum_{\ell\geq i}\frac{\rho_\ell}{\rho_j(1-V_j)}
    &\leq \frac{1}{1-t}
    \sum_{j=k+1}^ie^{-\gamma(i-j)}\sum_{\ell\geq i}e^{-\gamma(\ell-i)} \\
    &\leq \frac{(1-e^{-\gamma})^{-2}}{1-t}=\frac{1-t}{(1-2t)^2}.
  \end{align*}
  In particular, $B_2^+\leq (1-t)/(1-2t)^2$, which concludes the proof.
\end{proof}

\bigskip

{\noindent\textbf{Acknowledgements.} The second author would like to thank the
  \emph{\'Equipe de Probabilit\'es et Statistique de l'Institut de
    Math\'ematiques de Toulouse} for kind hospitality. The last author would
  like to thank Delphine \textsc{F\'eral} and Sandrine \textsc{P\'ech\'e} for
  interesting discussions on the extremal eigenvalues of symmetric
  non--central random matrices with i.i.d entries.}

\bigskip

\vfill

\begin{center}
  \begin{figure}[htb]
    \begin{center}
      \includegraphics[scale=0.45]{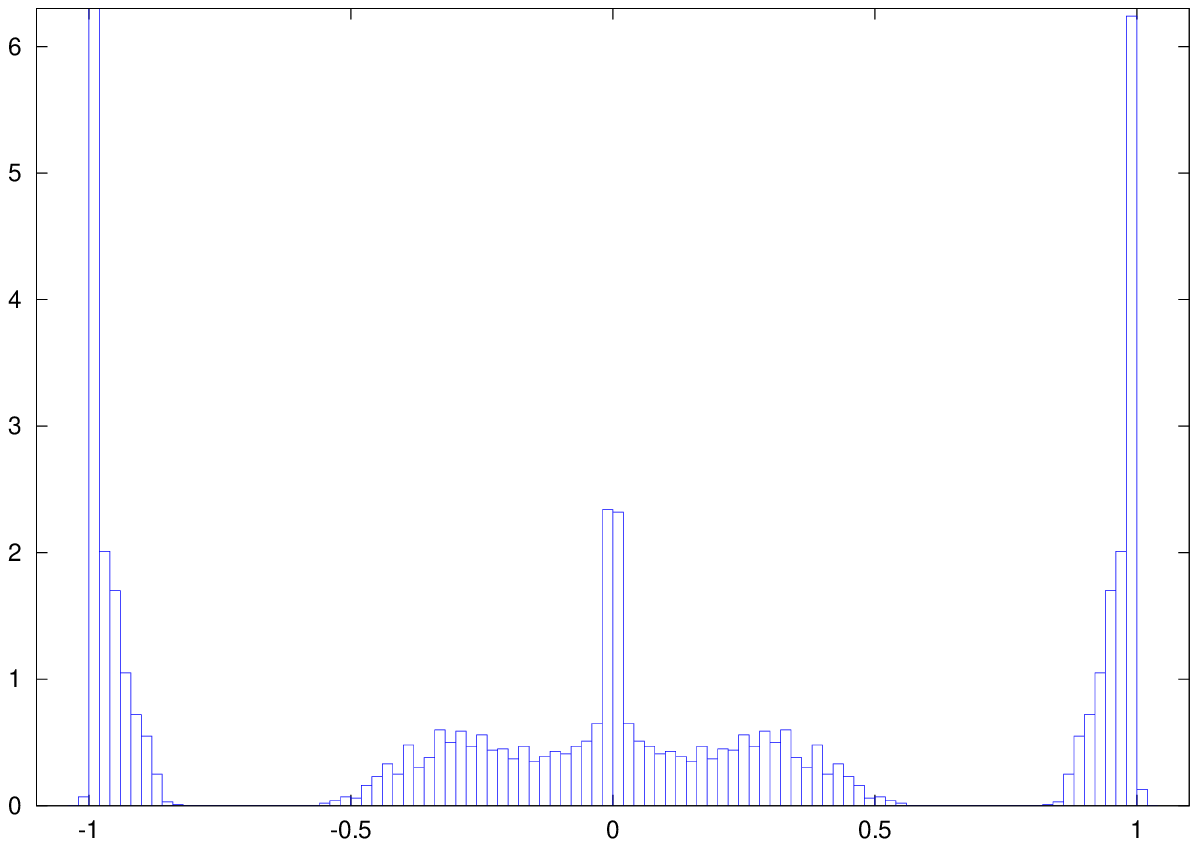}
      \includegraphics[scale=0.45]{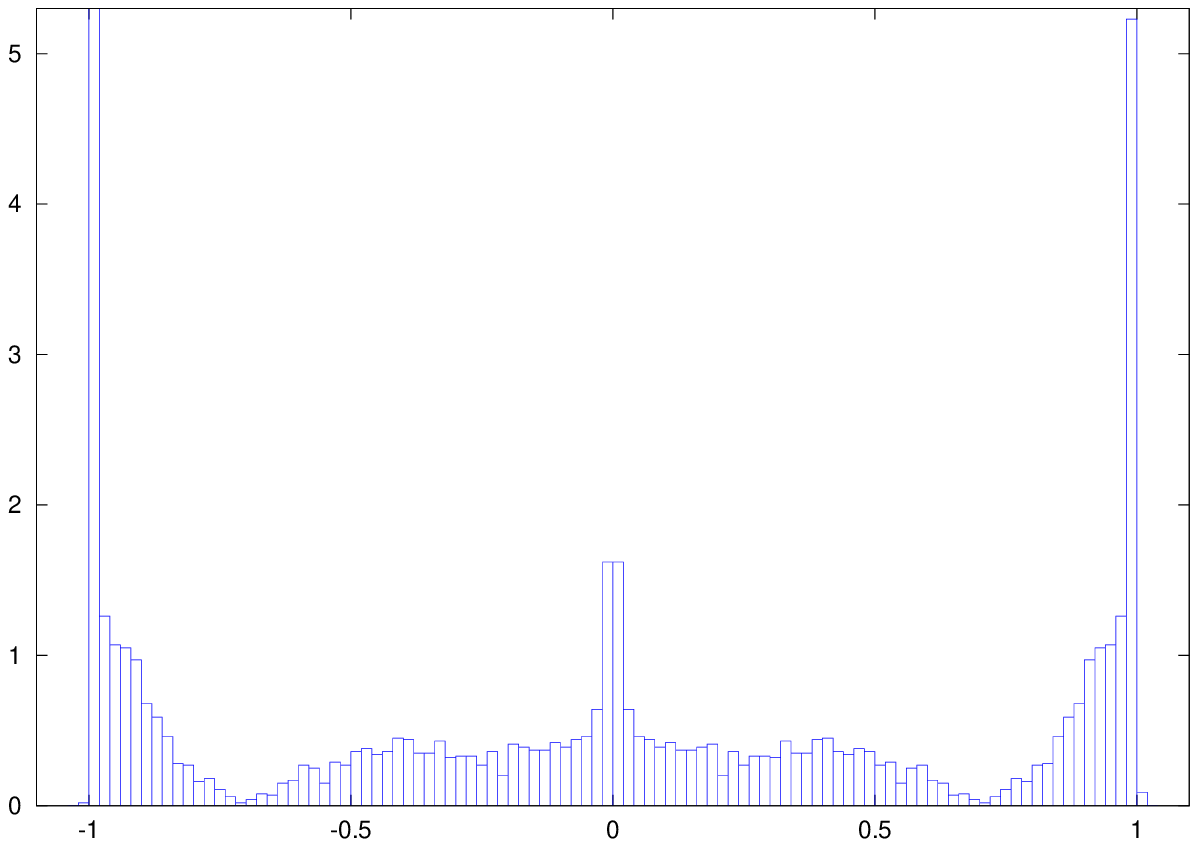}
    \end{center}
    \caption{Plots illustrating Corollary \ref{co:chain}. Each histogram
      corresponds to the spectrum of a single realization of $K$ with
      $n=5000$, for various choices of $\mathcal{L}$. From left to right
      $\mathcal{L}$ is the uniform law on $[0,t]\cup[1-t,1]$ for $t=1/8$,
      $t=1/4$. 
    }
    \label{fi:tridiaga}
  \end{figure}
\end{center}

\begin{center}
  \begin{figure}[htb]
    \begin{center}
      \includegraphics[scale=0.45]{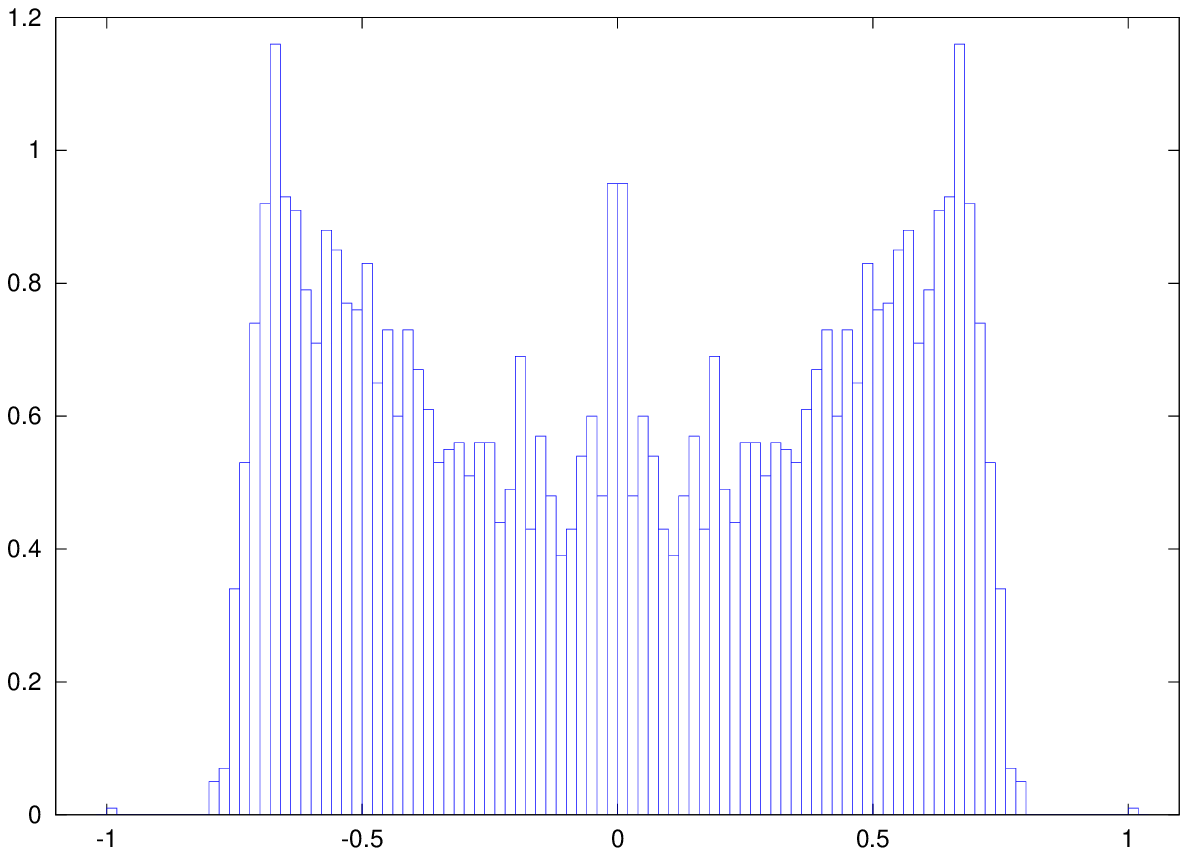}
      \includegraphics[scale=0.45]{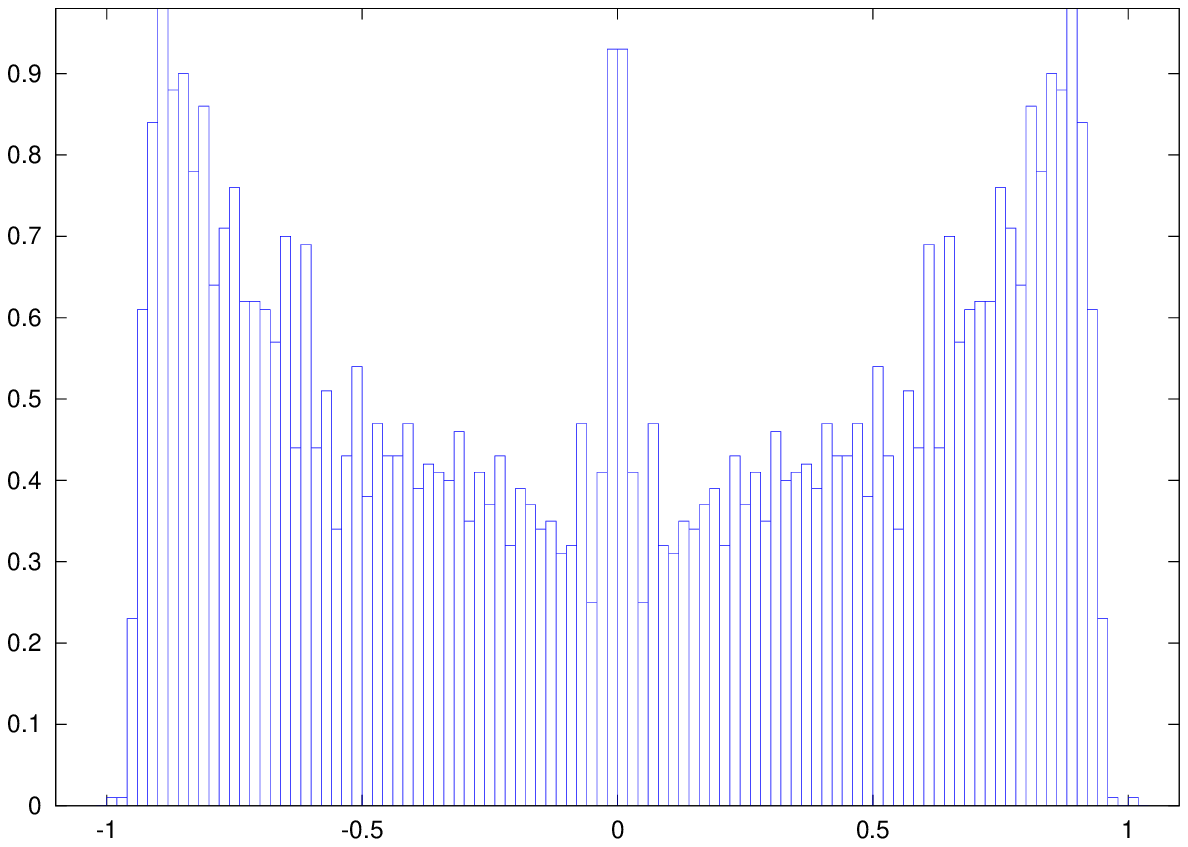}
      
      \includegraphics[scale=0.45]{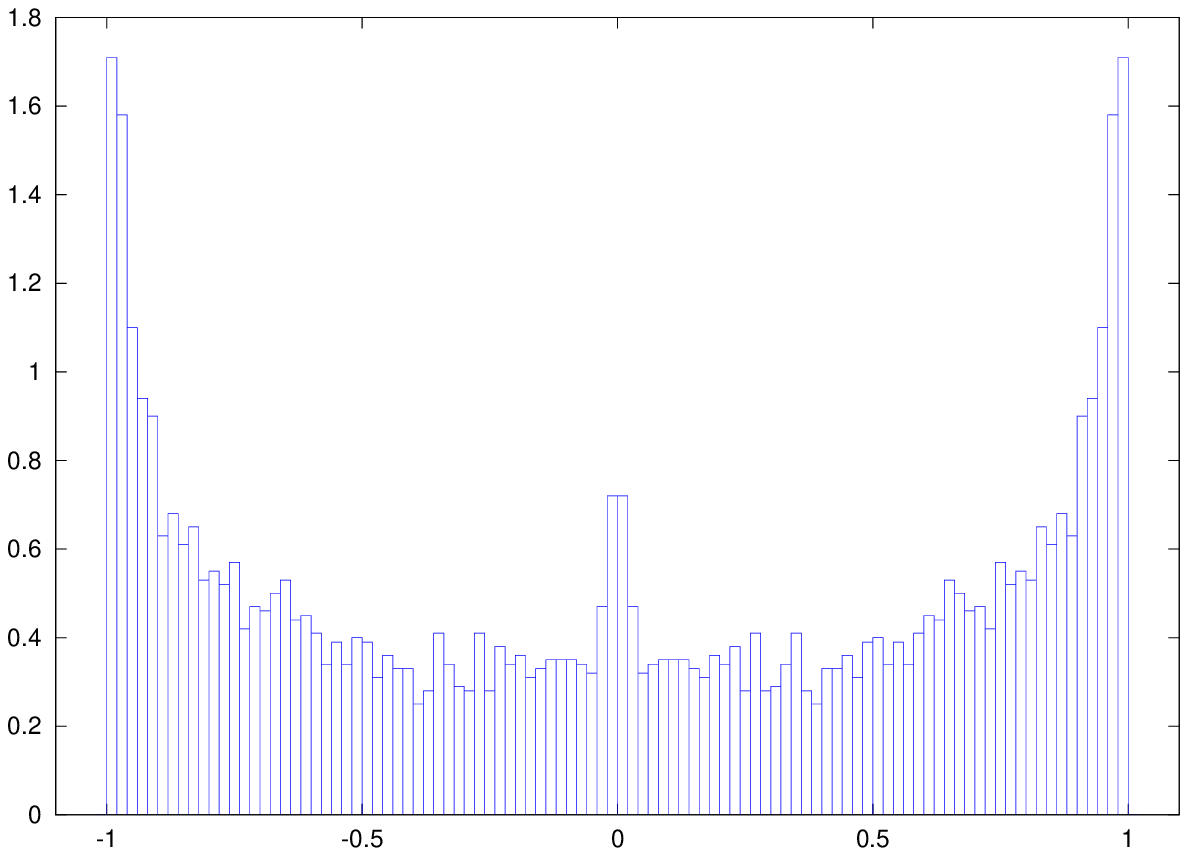}
      \includegraphics[scale=0.45]{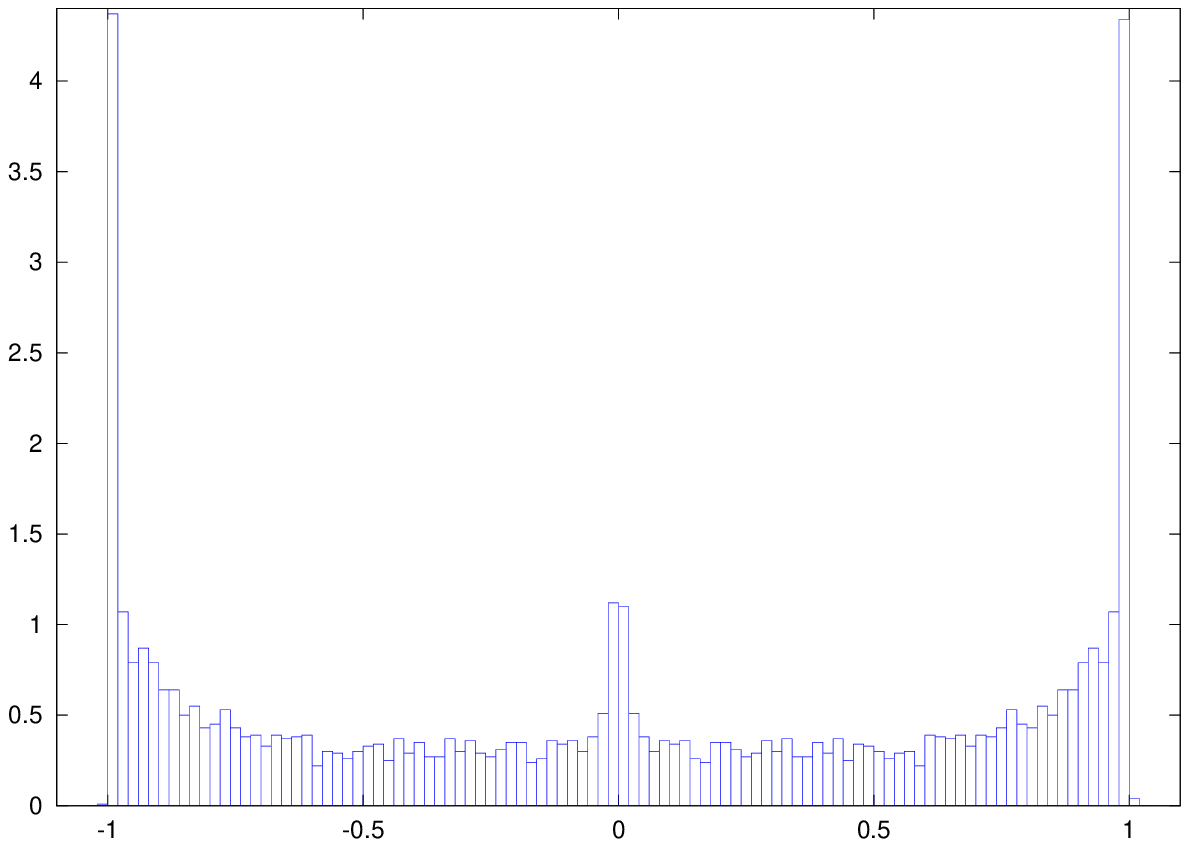}
    \end{center}
    \caption{Plots illustrating Corollary \ref{co:chain}
      and the second statement of Theorem \ref{th:chainedge}. Each histogram
      corresponds to the spectrum of a single realization of $K$ with
      $n=5000$, for various choices of $\mathcal{L}$. From left to right and
      top to bottom, $\mathcal{L}$ is uniform on $[0,t]$ with $t=1/8$,
      $t=1/4$, $t=1/2$, and $t=1$. 
    }
    \label{fi:tridiagb}
  \end{figure}
\end{center}

\begin{center}
  \begin{figure}[htb]
    \begin{center}
      \includegraphics[scale=0.45]{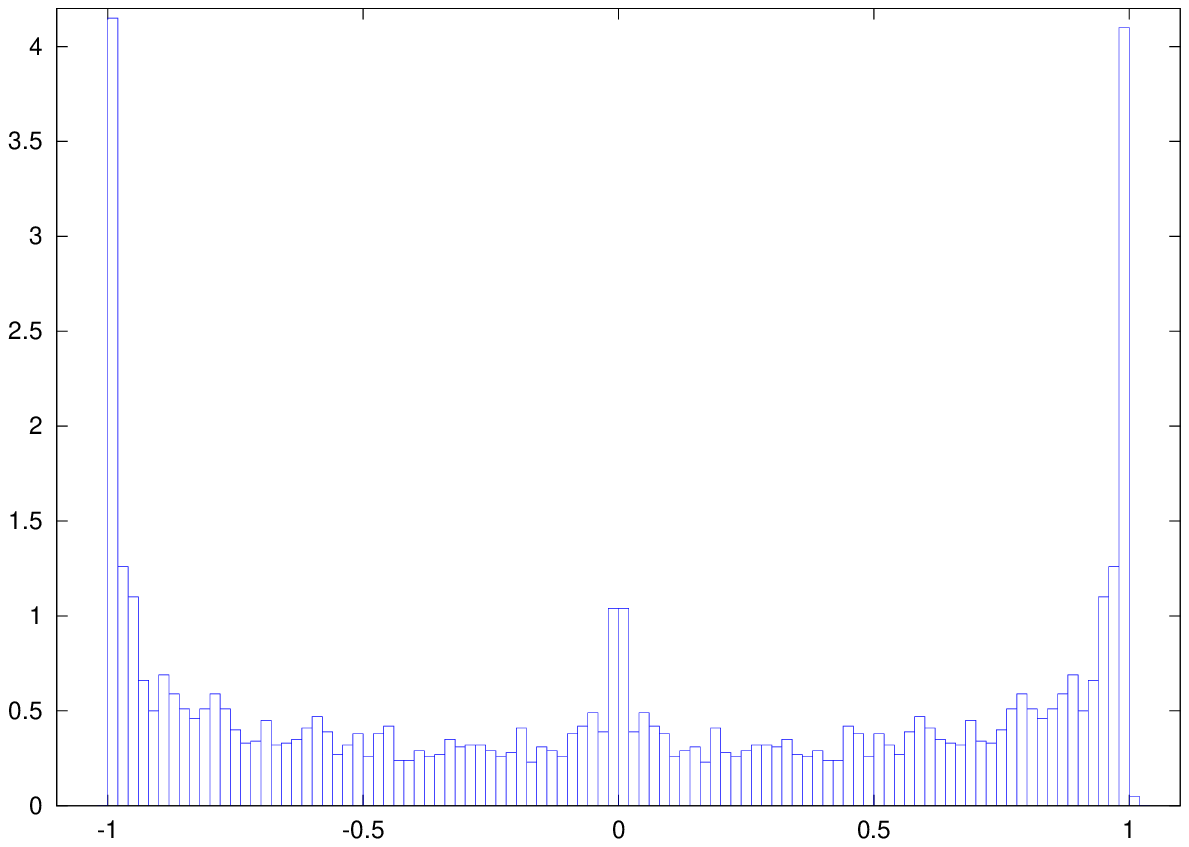}
      \includegraphics[scale=0.45]{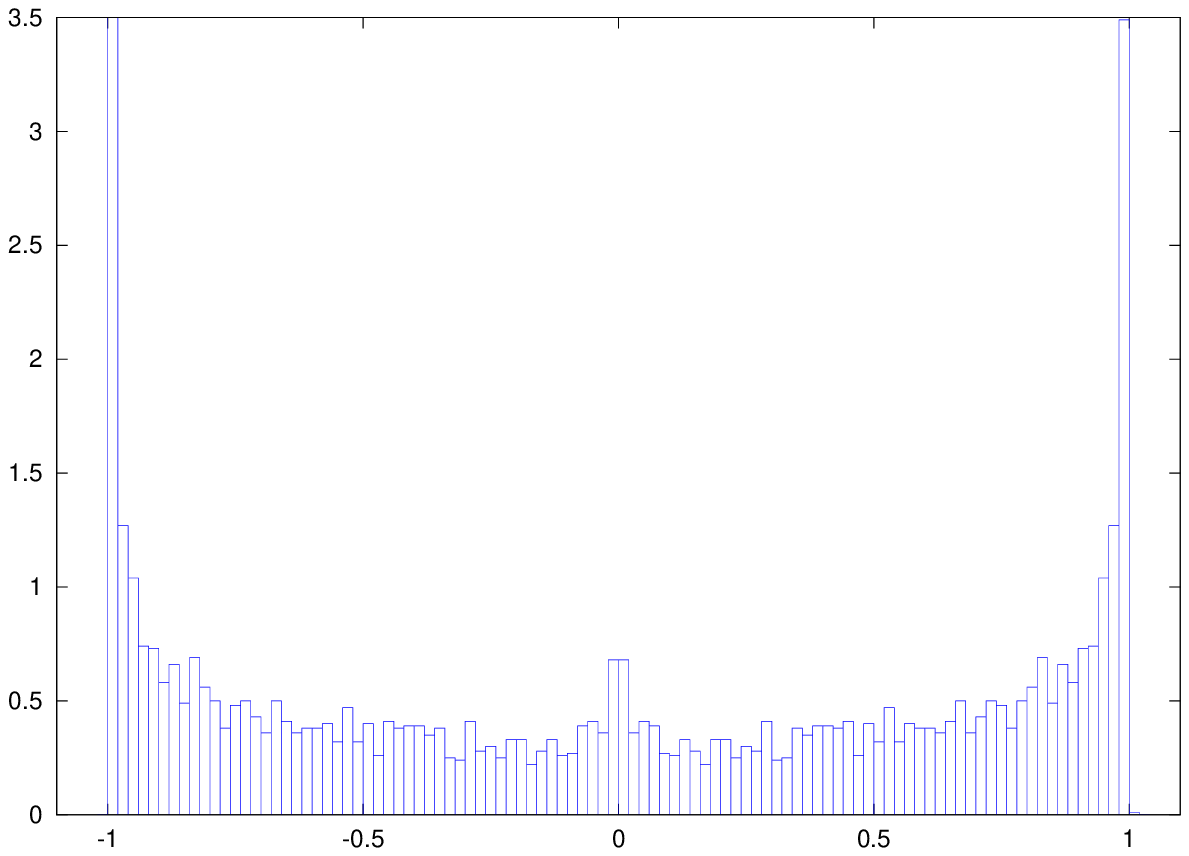}
      
      \includegraphics[scale=0.45]{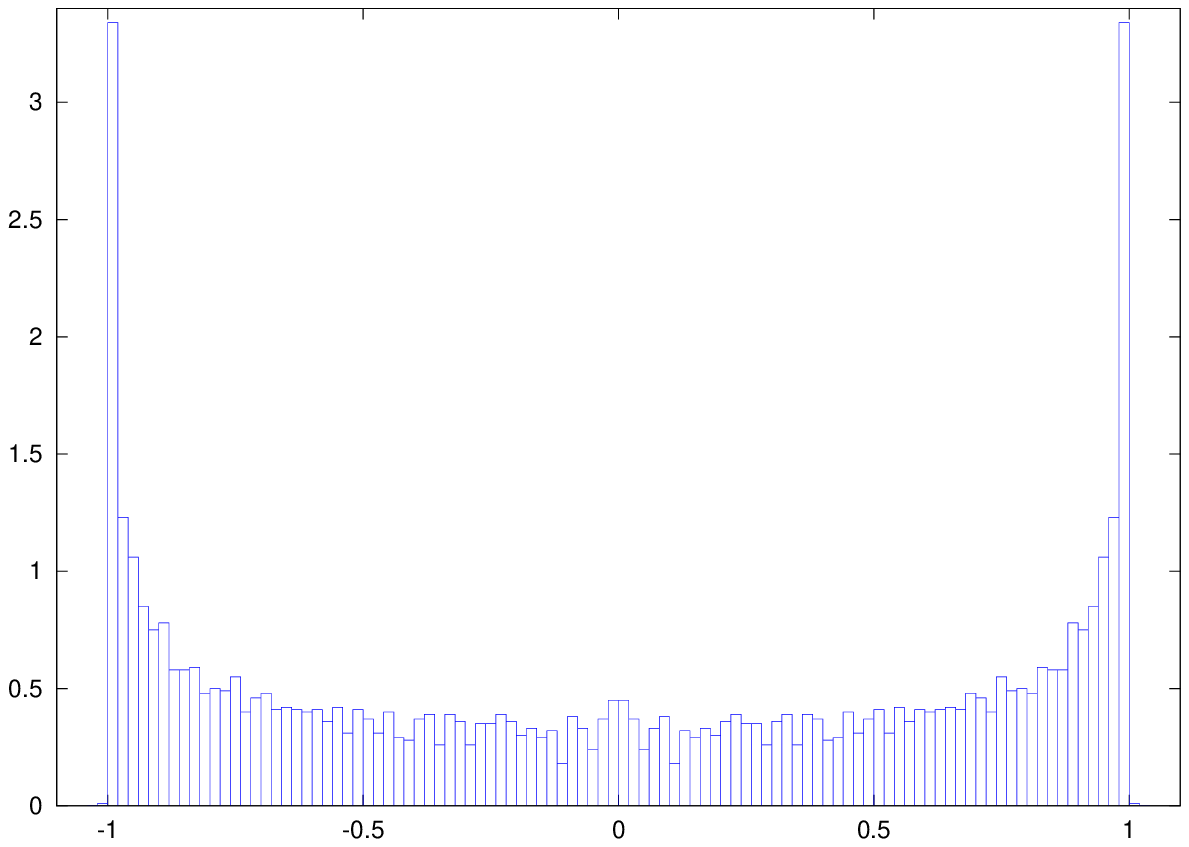}
      \includegraphics[scale=0.45]{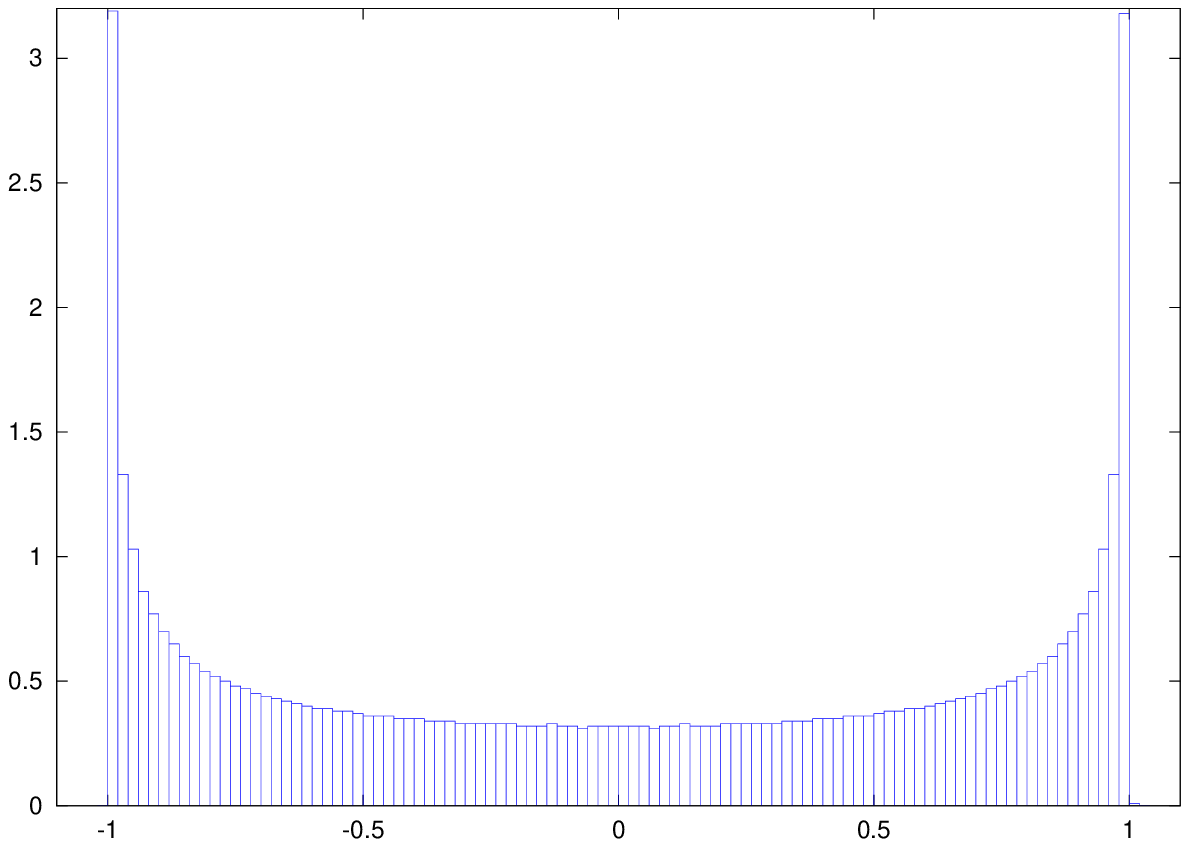}
    \end{center}
    \caption{Plots illustrating Corollary \ref{co:chain}. Each histogram
      corresponds to the spectrum of a single realization of $K$ with
      $n=5000$, for various choices of $\mathcal{L}$. From left to right and
      top to bottom, $\mathcal{L}$ is uniform on $[t,1-t]$ with $t=0$,
      $t=1/8$, $t=1/4$, $t=1/2$. The last case corresponds to the arc--sine
      limiting spectral distribution mentioned in Remark \ref{re:arcsine}.
    }
    \label{fi:tridiagc}
  \end{figure}
\end{center}

\bibliographystyle{amsplain}
\bibliography{bor-cap-cha}

\providecommand{\bysame}{\leavevmode\hbox to3em{\hrulefill}\thinspace}
\providecommand{\MR}{\relax\ifhmode\unskip\space\fi MR }
\providecommand{\MRhref}[2]{%
  \href{http://www.ams.org/mathscinet-getitem?mr=#1}{#2}
}
\providecommand{\href}[2]{#2}
\begin{thebibliography}{10}

\bibitem{anderson-gionnet-zeitouni}
G.~W. Anderson, A.~Guionnet, and O.~Zeitouni, \emph{An {I}ntroduction to
  {R}andom {M}atrices}, Cambridge {U}niversity {P}ress, 2009, to appear.

\bibitem{MR2548495}
A.~Auffinger, G.~Ben~Arous, and S.~P{\'e}ch{\'e}, \emph{Poisson convergence for
  the largest eigenvalues of heavy tailed random matrices}, Ann. Inst. Henri
  Poincar\'e Probab. Stat. \textbf{45} (2009), no.~3, 589--610. \MR{MR2548495}

\bibitem{MR1711663}
Z.~D. Bai, \emph{Methodologies in spectral analysis of large-dimensional random
  matrices, a review}, Statist. Sinica \textbf{9} (1999), no.~3, 611--677, With
  comments by G. J.\ Rodgers and J. W.\ Silverstein; and a rejoinder by the
  author.

\bibitem{MR958213}
Z.~D. Bai and Y.~Q. Yin, \emph{Necessary and sufficient conditions for almost
  sure convergence of the largest eigenvalue of a {W}igner matrix}, Ann.
  Probab. \textbf{16} (1988), no.~4, 1729--1741.

\bibitem{MR1235416}
\bysame, \emph{Limit of the smallest eigenvalue of a large-dimensional sample
  covariance matrix}, Ann. Probab. \textbf{21} (1993), no.~3, 1275--1294.

\bibitem{benarous-guionnet}
G.~Ben~Arous and A.~Guionnet, \emph{The spectrum of heavy tailed random
  matrices}, Comm. Math. Phys. \textbf{278} (2008), no.~3, 715--751.

\bibitem{MR1700749}
P.~Billingsley, \emph{Convergence of probability measures}, second ed., Wiley
  Series in Probability and Statistics: Probability and Statistics, John Wiley
  \& Sons Inc., New York, 1999, A Wiley-Interscience Publication.

\bibitem{MR2371333}
G.~Biroli, J.-P. Bouchaud, and M.~Potters, \emph{On the top eigenvalue of
  heavy-tailed random matrices}, Europhys. Lett. EPL \textbf{78} (2007), no.~1,
  Art. 10001, 5.

\bibitem{MR1956471}
D.~Boivin and J.~Depauw, \emph{Spectral homogenization of reversible random
  walks on {$\mathbb{Z}^d$} in a random environment}, Stochastic Process. Appl.
  \textbf{104} (2003), no.~1, 29--56.

\bibitem{MR1864966}
B.~Bollob{\'a}s, \emph{Random graphs}, second ed., Cambridge Studies in
  Advanced Mathematics, vol.~73, Cambridge University Press, Cambridge, 2001.

\bibitem{MR1890289}
E.~Bolthausen and A.-S. Sznitman, \emph{Ten lectures on random media}, DMV
  Seminar, vol.~32, Birkh\"auser Verlag, Basel, 2002. \MR{MR1890289
  (2003f:60183)}

\bibitem{bordenave-caputo-chafai-ii}
Ch.\ Bordenave, P.~Caputo, and D.~Chafa{\"{\i}}, \emph{Spectrum of large random
  reversible {M}arkov chains -- {H}eavy--tailed weigths on the complete graph},
  arXiv:0903.3528, 2009.

\bibitem{MR2152251}
A.~Bovier and A.~Faggionato, \emph{Spectral characterization of aging: the
  {REM}-like trap model}, Ann. Appl. Probab. \textbf{15} (2005), no.~3,
  1997--2037.

\bibitem{MR2370603}
\bysame, \emph{Spectral analysis of {S}inai's walk for small eigenvalues}, Ann.
  Probab. \textbf{36} (2008), no.~1, 198--254.

\bibitem{MR2166276}
S.~Boyd, P.~Diaconis, P.~Parrilo, and L.~Xiao, \emph{Symmetry analysis of
  reversible {M}arkov chains}, Internet Math. \textbf{2} (2005), no.~1, 31--71.

\bibitem{MR2206341}
W.~Bryc, A.~Dembo, and T.~Jiang, \emph{Spectral measure of large random
  {H}ankel, {M}arkov and {T}oeplitz matrices}, Ann. Probab. \textbf{34} (2006),
  no.~1, 1--38.

\bibitem{chafai-dme}
D.~Chafa\"\i, \emph{The {D}irichlet {M}arkov {E}nsemble}, Journal of
  Multivariate Analysis \textbf{101} (2010), 555--567.

\bibitem{cheliotis-virag}
D.~Cheliotis and B.~Virag, \emph{The spectrum of the random environment and
  localization of noise}, preprint, \texttt{arXiv.math:0804.4814}, to appear in
  Probability Theory and Related Fields, 2008.

\bibitem{MR920811}
P.~G. Doyle and J.~L. Snell, \emph{Random walks and electric networks}, Carus
  Mathematical Monographs, vol.~22, Mathematical Association of America,
  Washington, DC, 1984.

\bibitem{MR0228020}
W.~Feller, \emph{An introduction to probability theory and its applications.
  {V}ol. {I}}, Third edition, John Wiley \& Sons Inc., New York, 1968.
  \MR{MR0228020 (37 \#3604)}

\bibitem{MR637828}
Z.~F{\"u}redi and J.~Koml{\'o}s, \emph{The eigenvalues of random symmetric
  matrices}, Combinatorica \textbf{1} (1981), no.~3, 233--241.

\bibitem{MR1746976}
F.~Hiai and D.~Petz, \emph{The semicircle law, free random variables and
  entropy}, Mathematical Surveys and Monographs, vol.~77, American Mathematical
  Society, Providence, RI, 2000. \MR{MR1746976 (2001j:46099)}

\bibitem{MR0052379}
A.~J. Hoffman and H.~W. Wielandt, \emph{The variation of the spectrum of a
  normal matrix}, Duke Math. J. \textbf{20} (1953), 37--39.

\bibitem{MR1091716}
R.~Horn and C.~Johnson, \emph{Topics in matrix analysis}, Cambridge University
  Press, Cambridge, 1991.

\bibitem{MR2466937}
D.A. Levin, Y.~Peres, and E.L. Wilmer, \emph{Markov chains and mixing times},
  American Mathematical Society, Providence, RI, 2009, With a chapter by James
  G. Propp and David B. Wilson. \MR{MR2466937}

\bibitem{MR2386089}
J.-F. Marckert, \emph{One more approach to the convergence of the empirical
  process to the {B}rownian bridge}, Electron. J. Stat. \textbf{2} (2008),
  118--126. \MR{MR2386089 (2009a:62230)}

\bibitem{MR2129906}
M.~L. Mehta, \emph{Random matrices}, third ed., Pure and Applied Mathematics
  (Amsterdam), vol. 142, Elsevier/Academic Press, Amsterdam, 2004.
  \MR{MR2129906 (2006b:82001)}

\bibitem{MR1710983}
L.~Miclo, \emph{An example of application of discrete {H}ardy's inequalities},
  Markov Process. Related Fields \textbf{5} (1999), no.~3, 319--330.
  \MR{MR1710983 (2000h:60081)}

\bibitem{tetali-montenegro}
R.~Montenegro and P.~Tetali, \emph{Mathematical {A}spects of {M}ixing {T}imes
  in {M}arkov {C}hains}, Foundations and {T}rends in {T}heoretical {C}omputer
  {S}cience, vol. 1:3, Now Publishers, 2006.

\bibitem{MR2480789}
I.~Popescu, \emph{General tridiagonal random matrix models, limiting
  distributions and fluctuations}, Probab. Theory Related Fields \textbf{144}
  (2009), no.~1-2, 179--220. \MR{MR2480789}

\bibitem{MR1490046}
L.~Saloff-Coste, \emph{Lectures on finite {M}arkov chains}, Lectures on
  probability theory and statistics (Saint-Flour, 1996), Lecture Notes in
  Math., vol. 1665, Springer, Berlin, 1997, pp.~301--413.

\bibitem{MR2209438}
E.~Seneta, \emph{Non-negative matrices and {M}arkov chains}, Springer Series in
  Statistics, Springer, New York, 2006, Revised reprint of the second (1981)
  edition [Springer-Verlag, New York; MR0719544]. \MR{MR2209438}

\bibitem{MR2081462}
A.~Soshnikov, \emph{Poisson statistics for the largest eigenvalues of {W}igner
  random matrices with heavy tails}, Electron. Comm. Probab. \textbf{9} (2004),
  82--91 (electronic).

\bibitem{MR2198849}
A.-S. Sznitman, \emph{Topics in random walks in random environment}, School and
  {C}onference on {P}robability {T}heory, ICTP Lect. Notes, XVII, Abdus Salam
  Int. Cent. Theoret. Phys., Trieste, 2004, pp.~203--266 (electronic).
  \MR{MR2198849 (2007b:60247)}

\bibitem{tao-vu-cirlaw-bis}
T.~Tao and V.~Vu, \emph{Random matrices: {U}niversality of {E}{S}{D}s and the
  circular law}, preprint \texttt{arXiv:0807.4898 [math.PR]} to appear in the
  Annals of Probability, 2008.

\bibitem{MR1964483}
C.~Villani, \emph{Topics in optimal transportation}, Graduate Studies in
  Mathematics, vol.~58, American Mathematical Society, Providence, RI, 2003.
  \MR{MR1964483 (2004e:90003)}

\bibitem{MR2432537}
Van Vu, \emph{Random discrete matrices}, Horizons of combinatorics, Bolyai Soc.
  Math. Stud., vol.~17, Springer, Berlin, 2008, pp.~257--280. \MR{MR2432537
  (2009i:15034)}

\bibitem{Zakharevich}
I.~Zakharevich, \emph{A generalization of {W}igner's law}, Comm. Math. Phys.
  \textbf{268} (2006), no.~2, 403--414.

\bibitem{MR2071631}
O.~Zeitouni, \emph{Random walks in random environment}, Lectures on probability
  theory and statistics, Lecture Notes in Math., vol. 1837, Springer, Berlin,
  2004, pp.~189--312. \MR{MR2071631 (2006a:60201)}

\end{thebibliography}

\end{document}